\documentclass[a4paper]{article}
\usepackage{amsmath,amsfonts,amssymb}
\usepackage{graphicx, epsfig}
\newtheorem{theorem}{Theorem}[section]
\newtheorem{lemma}[theorem]{Lemma}
\newtheorem{cor}[theorem]{Corollary}
\newtheorem{prop}[theorem]{Proposition}
\newtheorem{defn}[theorem]{Definition}
\newtheorem{ques}[theorem]{Question}
\newtheorem{conj}[theorem]{Conjecture}
\newtheorem{exam}[theorem]{Example}
\newtheorem{rem}[theorem]{Remark}

\newtheorem{idea}[theorem]{Idea} 

\newenvironment{proof}[1][Proof]{\textbf{#1.} }
{\hfill\rule{0.5em}{0.5em}\medskip}
\newenvironment{proof*}[1][Proof]{\textbf{#1.} }{}

\def\Kerekjarto{Ker\'ekj\'art\'o } 

\def\Hajek{H\'ajek }
\def\TOP{{\sf TOP}}
\def\DIFF{{\sf DIFF}}
\def\epsilon{\varepsilon}
\def\WB{$\omega$-bounded }

\begin{document}


\title{Dynamics of non-metric manifolds
\\}

\author{
Alexandre Gabard and David Gauld\footnote{Supported by the Marsden Fund Council from Government funding, administered by the Royal Society of New Zealand.}
}

\maketitle

\vskip -0.5cm

\newbox\quotation
\setbox\quotation\vtop{\hsize 10cm \noindent

\footnotesize {\it Je ne crois donc pas avoir fait une
{\oe}uvre inutile en \'ecrivant le pr\'esent M\'emoire; je
regrette seulement qu'il soit trop long; mais quand j'ai voulu
me restreindre; je suis tomb\'e dans l'obscurit\'e; j'ai
pr\'ef\'er\'e passer pour un un peu bavard.}
%
(Henri Poincar\'e, 1895, introducing his {\it Analysis
situs}.)


}

\hfill{\hbox{\copy\quotation}}

\medskip

\smallskip
\newbox\abstract
\setbox\abstract\vtop{\hsize 11cm \noindent

\footnotesize \noindent\textsc{Abstract.} An attempt is made
to extend some of the basic paradigms of dynamics---from the
viewpoint of (continuous) flows---to non-metric manifolds.
}

\centerline{\hbox{\copy\abstract}}

{\small \tableofcontents}

\normalsize

\section{Introduction}\label{sec1}

The present paper, by far not having the intrinsic
{\it charism} of
Poincar\'e's M\'emoire, may share some of
the supposed
discursive defects---albeit
in the more annoying way that our
loquaciousness,
instead of reflecting a wealth of new
insights, resulted rather from a poor understanding of a
somewhat exotic subject-matter.
  At any rate, like Poincar\'e, we shall
  put ourselves at the
cross-intersection of the two
paradigms ``manifolds'' and ``dynamics''. Albeit
extensively studied since
and under his impulsion (1880--1913), the subject still
contains certain {\it places r\'eput\'ees jusqu'ici
inabordables}\footnote{To quote again---this time
loosely---Poincar\'e \cite[p.\,82]{Poincare_1892}.}, even
in seemingly anodyne situations.
Elusive, open problems
belonging to the genre are: {\it Hilbert's 16th problem} for
the number and mutual disposition of {\it cycles limites} of
polynomial vector fields in the plane, or the question as to
whether the 3-space ${\Bbb R}^3$ (or the 3-sphere) admits a
flow
with {\it all} orbits dense ({\it Gottschalk conjecture}
formulated
in 1958 \cite{Gottschalk_1958},
along the tradition of
Poincar\'e-Hadamard-Birkhoff-Morse-Hedlund, not to mention
Markoff \cite{Markoff_1931}, and reposed by Smale on many
occasions).

Beside
such difficult questions, a sizeable portion of
theory
is well buoyed, forming so to speak a main-stream
of knowledge. Our
task will
merely reduce
to selecting
among such well-oiled
mechanisms, those
capable of a
{\it tele-transportation} beyond the metric realm. To gain
some
swing, we shall
briefly loop-back at some early history, aiding---at least
fictionally---to
circumvent better the nature of the ``main-stream'' in
question.

During the 19th century the concept of space enjoyed a golden
reconfiguration producing the fruitful concept of
{\it manifolds} (Gauss, Lobatschevsky,
Riemann, etc). This involves the idea of a space locally
modelled over some ``flat'' number-space like ${\Bbb R}^n$.
Gradually the ``manifold'' idea came to its clear-cut
precision but perhaps only through
specialisation of the much
broader concept of a {\it topological space} (Hilbert 1902,
Fr\'echet 1906, Hausdorff 1914). Among the earliest
axiomatisation of manifolds we
count: Weyl 1913 \cite{Weyl_1913}
(with a triangulated influence of Brouwer), \Kerekjarto 1923
\cite[p.\,5]{Kerekjarto_1923} for pure {\it topological
manifolds} (=$C^0$-manifolds), Veblen-Whitehead 1931
\cite{Veblen-Whitehead_1931} for differential manifolds.

Beside this purely ``spatial'' development, physics
(typically
Newtonian mechanics)
set forth the
description of natural phenomena
evolving in time via
differential equations. The associated flows
became Poincar\'e's fleuron to launch the great qualitative
programme
(stability, instability, chaos, etc). Eventually, an easy
abstraction
allows one to
think about {\it flows} without reference to the differential
calculus, as a topological group action of the real line
${\Bbb R}$ over a certain topological space $f\colon {\Bbb
R}\times X \to X$. (This shift of viewpoint
occurs
by \Kerekjarto 1925 \cite{Kerekjarto_1925},  Markoff 1931
\cite{Markoff_1931}, and in full virtuosity by Whitney 1933
\cite{Whitney33}.)

Both notions ``manifolds'' and ``dynamics'' turned out to be
quickly
intermingled in a ``space-time''
companionship.
E.g., as early as 1839, Gauss in his {\it Allgemeine Theorie
des Erdmagnetismus} \cite[Artikel 12,
p.\,134--135]{Gauss_1839}, noticed that the speculation
that the earth might not have a unique magnetic
north pole would
ineluctably
create some other ``hybrid'' pole which is neither a north nor
a south pole\footnote{To quote Gauss more accurately ({\it
loc.\,cit.}): ``Von einigen Physikern ist die Meinung
aufgestellt, dass die Erde zwei magnetische Nordpole und zwei
S\"udpole habe: [\dots]---Sehen wir von der wirklichen
Beschaffenheit der Erde ab, und fassen die Frage allgemein
auf, so k\"onnen allerdings mehr als zwei magnetische Pole
existiren: es scheint aber noch nicht bemerkt zu sein, dass
sobald z.\,B. zwei Nordpole vorhanden sind, es nothwendig
zwichen ihnen noch einen dritten Punkt geben muss, der
gleichfalls ein magnetischer Pol, aber eigentlich weder ein
Nordpol noch ein S\"udpol, oder, wenn man lieber will, beides
zugleich ist.'' This, and other early history, are
surveyed in the
famous  paper of Dyck~\cite{Dyck_1888}. Recall also
the r\^ole of Listing, both for its link with
Maxwell and as a forerunner of ``homology''.}. This can of
course be recognised as an early form of the
{\it Poincar\'e-Hopf index theorem} (and the related
hairy-balls theorems). An intermediate link from Gauss to
Poincar\'e is the Kronecker index (1869) which allowed many of
the forerunner, e.g. Bohl 1904 \cite{Bohl_1904} to anticipate
by some years some of the
contributions of Brouwer.
The
Poincar\'e-Hopf index
theorem (and the allied Lefschetz fixed point theory) appear
as
prominent outcome of this era,
altogether incarnating one of the most basic link between the
shape of a space and its dynamics (thus, a good candidate to
keep in mind for tele-transportation). Without
rushing on this, recall also the consequence that {\it a
closed manifold
accepts a non-stationary flow
if and only
if its Euler
characteristic vanishes} (Hopf~\cite{Hopf_1927},
\cite{Alexandroff-Hopf_1935}).
[In
passing,
the reverse implication does {\it not} seem to have been
firmly
established for $C^0$-manifolds (more on this in
Section~\ref{Morse-Thom:sec}). The direct
sense
follows, of course, from Lefschetz's extension 1937
\cite{Lefschetz_1937} of
his theory
to the class of compact metric ANR's.]

A ``general'' manifold---defined merely via the locally
Euclidean desideratum---because
of its naked elegance is capable of
various forms of perversities,
which are traditionally brought into more respectableness
through  additional restrictions: e.g., the Hausdorff
separation axiom, metrisability of the topology, compactness,
differential structures, Riemannian metrics, etc.


The modest philosophy of our
text is that while the specialisation to metric manifolds is
essential for ``quantitative'' problems (e.g., the
classification of
$2$-manifolds\footnote{Worked out by M\"obius 1863
\cite{Moebius_1863}, Jordan, Klein, Weichold 1883, Dyck 1888
\cite{Dyck_1888}
Dehn-Heegaard 1907, etc. and in the non-compact case
\Kerekjarto 1923 \cite{Kerekjarto_1923}.}) there is some
respectable ``qualitative''
principles which are sufficiently robust to hold
non-metrically. Examples of this vein are the Jordan
separation theory, the Schoenflies theorem
(any circle bounds a disc).
Thus,
the Poincar\'e-Bendixson theory---relying on the
{\it sack argument}
acting as a trap for
trajectories on a surface
where Jordan separation holds true---also propagates
non-metrically.
From it and Schoenflies, one can draw a {\it hairy ball
theorem} for $\omega$-bounded\footnote{Recall that a {\it
$\omega$-bounded} space is one such that any countable subset
admits a compact closure. This point-set
concept when particularised to manifolds allows one to hope
recovering
 some of the ``finistic'' virtues of compact manifolds beyond the
metric realm, cf. e.g. Nyikos's
bagpipe theorem to be discussed below.}
simply-connected surfaces, yielding a wide
extension of the fact that the $2$-sphere cannot be brushed.
(By a {\it brush}, we shall mean a flow without stationary
points---following the terminology of Beck~\cite{Beck_1958}.)

Beside this pleasant propagation of certain robust paradigms,
it must
confessed that some other  fails dramatically.
A typical disruption of this
kind occurs to {\it Whitney's flows}.
First, the
natural
desideratum of attaching to a
brush its induced foliation works universally in class $C^1$
and topologically  in low-dimensions $\le 3$ (Whitney 1938
\cite{Whitney_1938}), but not in higher-dimensions in view of
wild $C^0$-actions \`a la Bing (cf. Chewning 1974
\cite{Chewning_1974}).
Next, this process admits a
reverse
 engineering,
 which {\it creates}
 a
flow-motion
compatible with a given one-dimensional orientable foliation
(Whitney 1933 \cite{Whitney33}). Thinking of such {\it
Whitney's flows}---in rough caricature---as
obtained by parameterizing leaves
by arc length indicates
a definite {\it metric} sensitivity. It is not surprising
therefore, that one can easily experiment
non-metric failures
(Propositions~\ref{Whitney's_flow-failure-on-plane_pipe} and
\ref{Whitney's_flow_disruption_on_Moore}) even when all leaves
are {\it short} (i.e., metric). In the same vein,
Hopf's issue that a vanishing Euler characteristic is
sufficient
for
a
brush
lacks
a non-metric counterpart (Remark~\ref{reverse-engineering}
discusses the example of the connected sum of two long planes
${\Bbb L}^2$).

Accordingly, results
from the classical theory can
be
sorted out under the following three headings depending on
their
ballistic when catapulted outside the metric stratosphere:

(1) {\it Stable theories and theorems (``passe-partout'' {\rm
in Grothendieck's
jargon})}:
those
sufficiently robust as
to hold non-metrically. Examples: Jordan, Schoenflies,
Poincar\'e-Bendixson, two-dimensional hairy-ball theorems,
{\it phagocytosis}, i.e. the aptitude for a cell (chart) to
engulf any
countable subset of a manifold, cf. Gauld
\cite{Gauld_2009}\footnote{In the compact case such
phagocytosis appears in the work of Morton Brown
\cite{Brown_1962}, and Doyle-Hocking.}. This consequence of
Morton Brown's monotone-cell-union theorem,
also turns out to have multiple dynamical repercussions, as we
shall see.

(2) {\it Unstable theorems and vacuous paradigms}: theorems
breaking down
outside the metric world (Example:
Whitney's flows,
Hopf's brushes when $\chi=0$,
Beck's technique for slowing-down flow lines); and paradigms
which do not
survive by
lacking any single non-metric representative: Lie group
structures, minimal flows, global parallelism (\`a la
Stiefel). [Of course we do {\it not} claim that those theorems
are less good that those of the first category (1), but rather
that their non-metric collapses
adumbrate
a deeper geometric
substance.]---And finally:

(3) {\it Chaotical (undecided) paradigms:}  principles
which as yet (under our fingers) could not be
ranked into one of the previous two headings.
Examples: Finiteness property for the singular homology of
$\omega$-bounded manifolds and Lefschetz fixed point
theorem, hairy-ball theorems for $\omega$-bounded manifolds
with
$\chi\neq 0$, existence of smooth structures in low-dimensions
$\le 3$ (Spivak-Nyikos question, ref. as in \cite{GaGa2010}),
existence of transitive flows on separable manifolds of
high-dimensions $\ge 3$ (\`a la
Oxtoby-Ulam~\cite{Oxtoby-Ulam_1941},
Sidorov~\cite{Sidorov_1968}, Anosov-Katok~\cite{Anosov_1974}).

For simplicity, we count as a subclass of (3) {\it truly
chaotical} ({\it undecidable}) results which are known to be
sensitive
on some
axiomatic beyond ZFC (Zermelo-Fraenkel-Choice). Example: {\it
perfect normality}, i.e. the possibility of cutting-out an
arbitrary closed set as the zero-locus of a real-valued
continuous function. [Work of M.\,E. Rudin, Zenor.]


\begin{figure}[h] \centering
    \epsfig{figure=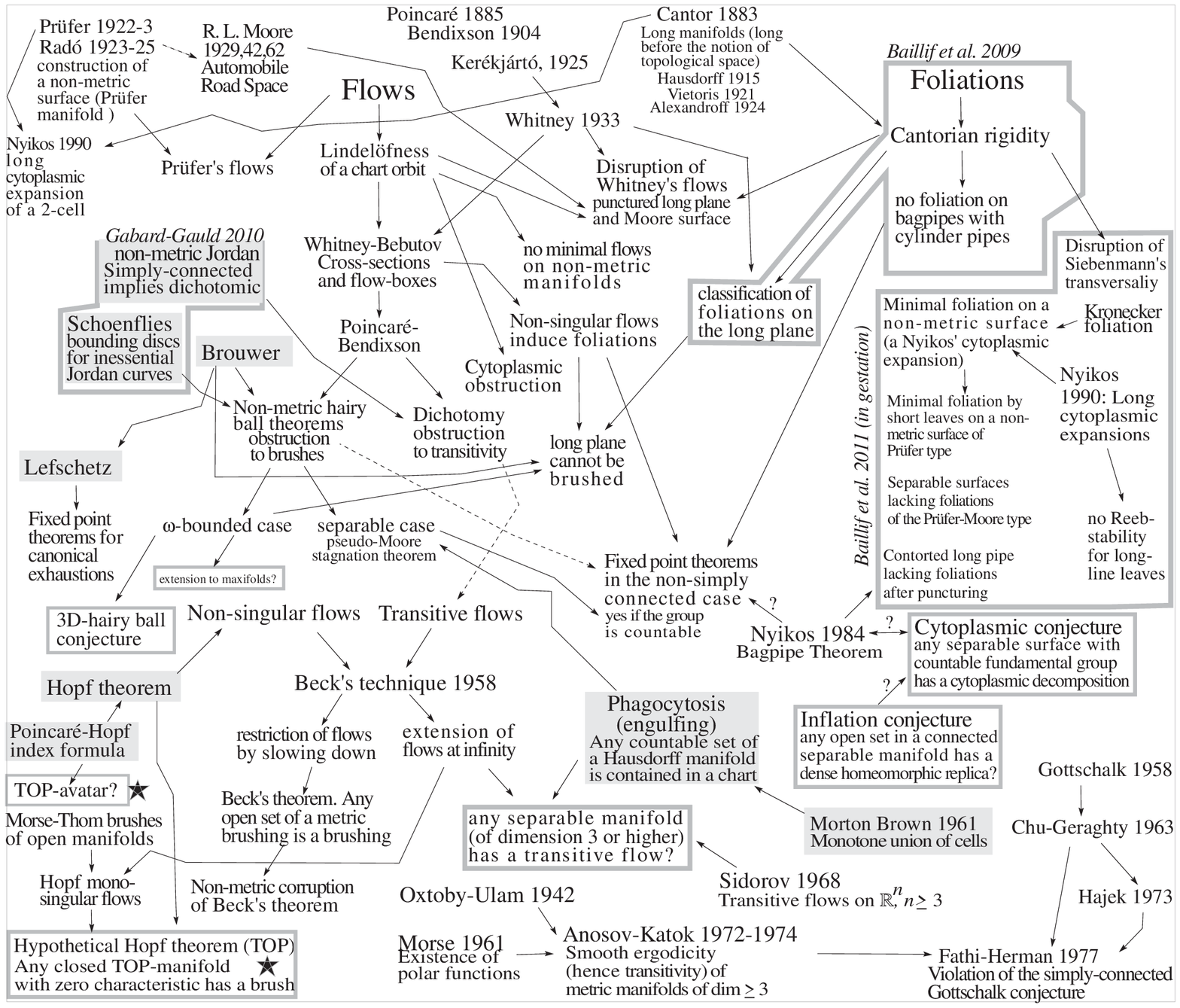,width=170mm}
\vskip-5pt\penalty0
\end{figure}

The by-standing synoptic diagram may help as a navigation
system;
it shows:

$\bullet$ Two islands delineated by frames,
packaging results
of previous papers by
the authors.
Since the present paper emphasises the viewpoint of flows, the
paper \cite{BGG} (concerned with foliations) is not an
absolute
prerequisite. The
note \cite{GaGa2010} (Jordan and Schoenflies in non-metrical
analysis situs) will be used in some arguments.

$\bullet$ Shaded regions marks classical
theorems that might
be adumbrative of certain
non-metric prolongations. Admittedly,
certain aspects of the non-metric theory of manifolds is just
a matter of transposing classical results
via transfinite repetition or by Lindel\"of approximation of
``small'' metric sub-objects. (The paper \cite{GaGa2010}
certainly
provides a good
illustration of this
reductionism.)

$\bullet$ Framed rectangles correspond to conjectures
delineating severe limitations in the authors knowledge. The
two starred frames correspond to purely metric questions, as
to whether  the Poincar\'e-Hopf index formula,
eventually also
the Hopf existence theorem for
brushes when $\chi=0$, generalise to $C^0$-manifolds (cf.
Section~\ref{Morse-Thom:sec} for some heuristics).


\subsection{Non-metric manifolds:
a short historiography}

Perhaps first, some few words looping back to the
sources
of non-metric manifold
theory.
The top of the iceberg emerged in Cantor's 1883
{\it Punktmannichfaltigkeiten}
\cite[p.\,552]{Cantor}, where the
{\it long ray} and
allied {\it long line}
were suggested.
A second generation of natural---indeed perfectly
geometric---examples
occurred to
Pr\"ufer and Rad\'o 1922/1925 \cite{Rado_1923},
\cite{Rado_1925}, as a byproduct of their investigations of
Weyl's treatment
of
{\it Die Idee der Riemannschen Fl\"ache}~\cite{Weyl_1913}.
A similar
vision occurred, seven years later in 1929,
to R.\,L. Moore in the form of an {\it Automobile Road Space}
(see the report by F. Burton
Jones~\cite{Jones_F._Burton_1997}\footnote{Quoting from  F.
Burton Jones~\cite{Jones_F._Burton_1997}: ``{\it Coming home
from the ``Boulder meeting'' in the summer of 1929, Moore
discovered his Automobile Road Space. [It] is an example of a
nonseparable complete Moore space which is a $2$-manifold.}''
This fancy name corresponds to what is
nowadays commonly termed
the {\it Pr\"ufer surface} or the {\it Pr\"ufer manifold} (cf.
Rad\'o 1925 \cite{Rado_1925}, Carath\'eodory 1950
\cite{Caratheodory_1950}, Nevanlinna 1953
\cite{Nevanlinna_1953}, Calabi-Rosenlicht 1953
\cite{Calabi-Rosenlicht_1953}, Ganea 1954
\cite{Ganea_1954}).}),
thereby
rediscovering---apparently independently---the
example of Pr\"ufer.
Moore's contribution (published only in
1942~\cite{Moore_1942}) is
a
certain ``twist''
in the Pr\"ufer construction producing  separable
simply-connected examples.
A~noteworthy feature of both the Pr\"ufer and Moore
constructions is that they are not isolated
specimens, but rather more
``fabrics''
engendering a variety of civilised, easy-to-visualise
examples. From the set-theoretical
viewpoint, the real
eclosion of the subject---yet
another
heritage of R.\,L. Moore's School\footnote{From the dynamical
viewpoint we already alluded to
R.\,H. Bing's impact---via wild topology---on a fundamental
question of Whitney in 1933 (on cellular cross-sections),
implying a radical divorce (already in the metrical realm, of
course) between the
topological and smooth approach to dynamical systems.
A similar divorce occurred earlier with (discrete)
transformation groups,
(recall Bing's involution of the $3$-sphere (1952)
\cite{Bing_1952}). More on this in
Section~\ref{Whitney's_flow}.}---is incarnated by the
contributions of M.\,E.
Rudin, Zenor. The 1984 paper
of
Nyikos~\cite{Nyikos84} is
the best initiation to the vertiginous
depth
of the non-metric universe (even in the $2$-dimensional,
simply-connected setting).
It
also
achieves a
subtle balance
of point-set versus
combinatorial
methods,
culminating to the
{\it bagpipe theorem}, showing that the subclass of
$\omega$-bounded surfaces behaves
like the familiar compact $2$-manifolds, 
save for the presence of {\it long pipes} emanating out from
the {\it bag}, while travelling at such
sidereal distances
as to
violate any metrication.
Those pipes could be thought of as circle-bundles over a
closed long ray, yet their real structure is in general
somewhat more mysterious. In particular they do not
necessarily admit a {\it canonical} exhaustion by compact
bordered cylinders. This plague of ``wild pipes'' will
cause us some
troubles, when attempting to
tele-transport the Lefschetz fixed point theory.

For more intelligibility and to the convenience of the
(non-specialised) reader,
let us recall that the {\it bordered\footnote{Following
Ahlfors-Sario, we employ {\it bordered manifold} as a synonym
of ``manifold-with-boundary'', which seems to us better than
the ``bounded manifolds'', used e.g. by J.\,H.\,C.
Whitehead---avoiding any conflict with ``$\omega$-bounded''.}
Pr\"ufer surface}, $P$, can be thought of as the open
upper-half plane
$H={\Bbb R}\times {\Bbb R}_{>0}$ plus some ideal points
materialised by rays rooted on the horizontal boundary line
$\{y=0\}$ and pointing
into $H$. All this data can be
naturally topologised, to produce a certain {\it bordered}
surface, $P$, whose interior is an open $2$-cell and whose
boundary
splits into a ``continuum'' $\frak c = {\rm card} \Bbb R$  of
components each homeomorphic to the real line. Thus, faithful
to the automobile jargon, this Pr\"ufer surface, $P$,
resembles a windscreen with a continuum of wiper (not just
two)
each rooted at the bottom of the screen. With this picture, it
is easy to visualise a non-singular flow on
$P$, akin to a windscreen wiper motion (cf.
Figure~\ref{Artist_views}, left hand-side). Besides, the {\it
Moore surface} is
just the quotient of the bordered Pr\"ufer surface, $P$, by
gluing each of its boundary components via the identification
$x\sim -x$ (compare Figure~\ref{Artist_views}, right
hand-side). The figure also
shows two other $2$-manifolds naturally deducible from the
bordered $P$, namely a collared version $P_{\rm collar}=P \cup
(\partial P \times [0,\infty))$ (which turns out to be the
same as the original Pr\"ufer surface described in Rad\'o 1925
\cite{Rado_1925}), plus its {\it double} $2P=P\cup P$, i.e.,
the gluing of $P$ with a replica of itself
(compare Calabi-Rosenlicht
1953~\cite{Calabi-Rosenlicht_1953}).

\begin{figure}[h] \centering
    \epsfig{figure=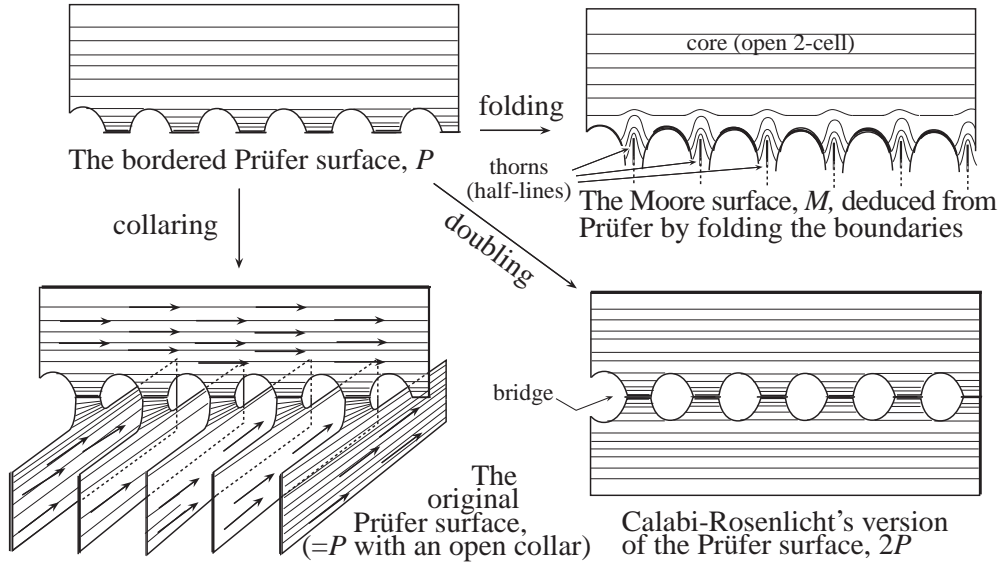,width=130mm}
  \caption{\label{Artist_views} Artist views of the Pr\"ufer
  and Moore surfaces (with galvanic currents)}
\vskip-5pt\penalty0
\end{figure}

\subsection{Dynamics of flows: overview of
results}

The issue---that many paradigms of dynamics holds true
non-metrically---has a
very simple
origin, rooted in the ``shortness'' of time,
modelled by the real line ${\Bbb R}$ (Dedekind-Cantor
continuum).
It implies, the orbit $f({\Bbb R}\times U)$ of any chart $U$
under
a flow
to be Lindel\"of, hence metric (Urysohn). Inside this metric
flow-invariant subspace, one can draw {\it cross-sections} and
{\it flow-boxes} (Whitney-Bebutov theory), yielding a
%
straightening of the motion in the vicinity of any
non-singular point.
One can then
ape the classical Poincar\'e-Bendixson theory,
and establish fixed-point theorems for flows (non-metric hairy
ball theorems)
under
weak point-set
assumptions (like $\omega$-boundedness, and later
separability),
plus
simple-connectivity.

Our
broad tolerance
for {\it non-metric} manifolds prompts the question of why we
are not
playing with longer groups,
e.g., the long line ${\Bbb L}$ as a model of time. Arguably,
any such model,
to deserve really the name, should
at least carry the structure of a topological group. In this
respect, a theorem of Garrett Birkhoff and Kakutani (1936)
\cite{Birkhoff_1936}
says that {\it a first countable topological group is
metrisable}; impeding
non-metric manifolds
entering
the
arena of topological ({\it a fortiori} Lie) groups.
Thus, dynamics may allow big spaces but is inherently
limited to short times.
(Here, foliations are more flexible, as leaves can easily
stretch into longness.)

Maybe the almost subconscious appeal of (continuous) flows
relies in part---beside the
  Kriegspiel or varied physico-chemical  interpretations---on
the tautological observation that the real line ${\Bbb R}$ is
the building brick of {\it any} manifold theory (whether
metric or not); a flow offering
thereby an introspection of
the manifold via its architectonic constituent.

Special attention was given---somewhat parallel to the
interest aroused by the ergodic hypothesis---when much of the space
is explored by starting from a {\it definite} resp. {\it any}
initial position. This leads to the classical notions of {\it
transitive}, resp. {\it minimal} flows as those having at
least {\it one} (resp. {\it all}\,) orbits dense in
phase-space. The
paradigmatic climax of minimal flows,
easily generated on tori (Kronecker), still leads to deep
questions like the {\it Gottschalk conjecture}
\cite{Gottschalk_1958}, already mentioned. (Assume
a
somniferous Riemannian
oracle saying
that  closed {\it positively curved}
manifolds lack a minimal
flow; then beside implying Gottschalk in {\it all\,}
dimensions, the result of Fathi-Herman (1977) would imply that
products of odd-spheres, e.g. ${\Bbb S}^3\times{\Bbb S}^3$,
lack positive curvature,
cracking partially
H. Hopf's puzzle from the 1930's.)
Surprisingly, the Gottschalk
inquiry takes a much simpler
{\it
tournure}
in the non-metric world: the Lindel\"ofness of any chart-orbit
$f({\Bbb R}\times U)$
{\it rules out} the  existence of any minimal
flow on {\it all} non-metric manifolds (in a single stroke!).
 This
fits into the picture
that non-metric manifolds cannot
concentrate
too
rich structures:  no group structures compatible with their
topology (Birkhoff-Kakutani), no minimal systems,
nor for instance a global {\it Fernparallelismus} in the sense
of Stiefel\footnote{Recall the argument of J.\,A.
Morrow~\cite{Morrow_1969}: a trivialisation of the tangent
bundle allows one to introduce a Riemannian metric which in
turn will metricise the manifold topology, in the large.}.

Alas, all these
limitations do not close the subject, as  despite not having
all the symmetry perfection of the metric theory, we shall
attempt to argue that
there is still a rich and easy-to-experiment ``geometry'' of
flows on non-metric manifolds\footnote{Of course, the issue
that some non-metric manifolds are also capable of a rich
 ``geometry'' is by no mean a new age philosophy;
recall Calabi-Rosenlicht's solution
\cite{Calabi-Rosenlicht_1953} of Bochner's conjecture
(existence of complex-analytic structures on certain
non-metric manifolds of the Pr\"ufer type;
question also implicit in
Carath\'eodory~\cite[p.\,94]{Caratheodory_1932}.)}.
Since minimal flows are too
much demanded,
we
switch
attention to the weaker notions of non-singular flows (alias
{\it brushes})
resp. transitive flows and ask: {\it Which (non-metric)
manifolds
supports
them?}

Answers can be obtained by a mixture of combinatorial and
point-set
methods. Clearly, a transitive manifold deserves to be
separable (rational-times of a dense orbit); yet not
necessarily metric (Example~\ref{Kronecker-Pruefer}
considers a Kronecker flow on the torus, suitably Pr\"uferised
along a portion of orbit).

For metric surfaces, the
standard obstruction to transitivity is {\it dichotomy} (i.e.,
any embedded circle disconnects the surface). This is the
classical inference of Jordan separation on
Poincar\'e-Bendixson, which remains
activated non-metrically
(Lemma~\ref{dicho_implies_intransitive}). The stronger
Schoenflies property, to the effect that any circle bounds a
disc, acts as an obstruction to brushes, provided the surface
is $\omega$-bounded
(Theorem~\ref{omega-bounded_Hairy-ball-thm}).

Another noteworthy obstruction to brushes
involves
merely general topology (viz. Lindel\"ofness):  assume the
flow $f$ on $M$ admits a ``small'' (viz. Lindel\"of) {\it
propagator} (i.e., a subset $\Sigma\subset M$ such that the
restricted flow-map $f\colon{\Bbb R} \times \Sigma \to M$ is
surjective) then the image $M$
is Lindel\"of.

This simple fact
identifies many
surfaces (e.g., the Moore surface) lacking any brush; and more
generally those $n$-manifolds admitting a {\it cytoplasmic}
(or {\it core-thorn}) {\it decomposition}, as defined in
Section~\ref{CTD:Section}. The easy
argument is best visualised on the Moore surface, whose very
specific
morphology---consisting of a ``core'' (the open half-plane of
Pr\"ufer), into which
many thin ``thorns'' (semi-lines) are sticking in
(compare Figure~\ref{Artist_views})---implies readily the core
to be a propagator (under any brush).
In contradistinction,
when this small propagator obstruction is vacuous, we are
frequently able to construct
brushes on non-metric surfaces typically those of Pr\"ufer
type (Proposition~\ref{Pruefer_flow}); corroborating thereby
the aforementioned intuition of the windscreen wiper motion
(depicted on Figure~\ref{Artist_views}).

The
inaptitude of the Moore surface to
``brush'', leads
to the question, if a surface sharing {\it abstractly} its
most
distinctive topological traits
(namely simply-connectedness, separability, non-metrisability
and boundaryless) can support a brush.
Theorem~\ref{hairy_ball_thm:separable} proposes
a
negative answer,
positively interpretable as a hairy ball theorem for this
class of {\it pseudo-Moore} surfaces. Like the
$\omega$-bounded hairy ball, this dual separable version
derives from the same
ingredients (non-metric Schoenflies, Brouwer, and
Poincar\'e-Bendixson), plus an extra-quick
owing to the {\it phagocytosis principle} (\`a la Morton
Brown).

A certain ``duality'' seems to relate the paradigms of
``$\omega$-boundedness'' and ``separability'',
at different levels. First, the conjunction of both properties
forces compactness. This is why, non-metrically, they
represent two {\it totally} disjoint
streams of forces.
Second, the
analogy goes further than
the common
hairy-ball theorem, for
$\omega$-boundedness
is crystallized---not to say immortalized (at least in
$2$-dimensions)---into the bag-pipe decomposition of Nyikos,
while it seems rather likely that separability
relates to what we just called cytoplasmic
decompositions. Yet, this is not completely true as
exemplified by the separable doubled Pr\"ufer surface, $2P$,
which lacks a cytoplasmic decomposition (because it has a
brush). Thus, separability alone is not enough to have a
cytoplasm, but restricting the fundamental group to be trivial
(or even countable) might be sufficient? Such a cytoplasmic
structure theory might represent a certain interest, yet we
shall not address this question further. Its dynamical
consequence would be that {\it non-metric separable surfaces
with countable fundamental groups lack a brush}. Albeit,
derived
via a blatantly hypothetical route, this turns out
to be a trivial consequence of the separable hairy ball
theorem
(Corollary~\ref{separable_hairy_ball_pi_1_ctble:cor}).
Since we slightly deviated in
the
topological register, it
seems also
opportune to notice that $\omega$-boundedness, indeed the
weaker sequential-compactness, implies a form of
maximality (akin to the one of
closed manifolds),
effecting that such manifolds are {\it inextensible} (or {\it
maxifolds}); i.e., they cannot be embedded in a larger
connected manifold of the same dimensionality. (This follows
at once---via a clopen argument---
from the invariance of the domain, which ensures openness of
the image.)

Summing up, one sees---especially in two-dimensions---that the
basic paradigms of dynamics (Poincar\'e-Bendixson, Brouwer's
fixed-point theorem) plus the idea of small propagation permit
a fairly
accurate answer to the question of which manifolds admits a
brush resp. a transitive flow. The slightly ``botanical''
Section~\ref{Geography} is
adumbrative of the exhaustiveness of the theoretical
obstructions listed so far, by constructing surfaces with
prescribed topology and dynamics. Next, what about higher
dimensions?

Here our results get more fragmentary, yet some positive
things happen. For instance one is tempted to transplant the
Euler obstruction $\chi\neq 0$ to the existence of brushes
from the compact to the non-metric realm. This can be achieved
via the Lefschetz fixed-point theorem,
provided some control is put on the growing mode of the
manifold via so-called {\it canonical exhaustions}
(Proposition~\ref{Lefschetz-fppf}). Yet the full punch would
be a truly non-metric Lefschetz theory for $\omega$-bounded
manifolds materialised by the following optimistic
conjectures:
---(1) The singular homology of such a manifold is finitely
generated.
---(2) The non-vanishing of the Lefschetz number of a map is a
sufficient condition for the existence of a fixed point.
%
%
%
Recall that Jaworowski 1971 \cite{Jaworowski_1971} proves a
Lefschetz fixed point theorem for any (metric or not)
manifolds, but only for {\it compact maps}, i.e. those with
relatively compact image.

Finally, we shall briefly
address the topic of transitive flows.
By their very definition, manifolds are locally Euclidean,
allowing one to stretch about any point an
open set homeomorphic to the number-space, ${\Bbb R}^n$.
In many cases (spheres, tori, etc.)  such a chart may be
inflated until to cover a sizeable (indeed dense) portion of
the manifold.
The {\it phagocytosis principle}---to the effect that every
countable subset of a manifold, whether metric or not, is
contained in a chart---shows this to be a general feature of
separable manifolds (also when non-metric, e.g., the doubled
Pr\"ufer $2P$).
In view of this, and the
technique of
Beck (allowing one to ``extend'', after a suitable
time-change, a flow given on a small space to a larger one),
one might hope to construct transitive flows on (m)any
separable manifolds of dimension $\ge 3$.
Unfortunately, we failed to reach serious conclusions in that
direction, either by removing the parenthetical ``(m)'' of
``many'' to make it an ``any'' (or by locating a
counterexample). In other words, the well-known issue---that
in dimensions $\ge 3$ {\it all} metric manifolds are
transitive (Oxtoby-Ulam, Sidorov, Anosov-Katok)---remains
undecided for non-metric (separable) manifolds.

In conclusion, it seems that non-metric manifolds
split into two types of populations, {\it civilised}
(metric-like)
against
wild {\it barbarians}.
The plague of wild pipes impeded us to formulate a universal
$\omega$-bounded Lefschetz theory, and some separable
manifolds with a wild topology at infinity (outside a
phagocytosing dense chart) might troubleshoot the Beck
technique. Of course, in every-day practice
one mostly interacts
with civilised examples (of the Cantor, Pr\"ufer or Moore
type), yet the barbarians exist---as reported by some advanced
sentinels (e.g. Nyikos \cite{Nyikos84})---and potentially
causes troubles to a naive-minded propagation of paradigms
like those of Lefschetz or Beck. The suspense is {\it
intact}\,! (As a very vague guess, the duality discussed above
might suggest that barbarians are equi-distributed in both
classes $\omega$-bounded vs. separable, so that a failure by
Lefschetz would imply a failure by Beck, and vice versa?)

\section{Flows versus foliations}\label{Flows-vs-Foliations}

Most of this Section~\ref{Flows-vs-Foliations} is a survey of
metric results, with straightforward non-metric extensions
afforded by the chart orbit trick (viz. its Lindel\"ofness).
Thus, the reader not primarily interested in foliations, but
merely in flows (easier to define, albeit
of a wilder transverse nature) can skip
it, and move forward to Section~\ref{Flows:Section}, and
refers back to it when necessary.

A ($C^0$) {\it flow}
is a continuous action $f\colon {\Bbb R} \times X \to X$ of
the additive reals on a certain topological space $X$.
Each map $f_t$ defined by $f_t(x)=f(t,x)$ is a homeomorphism
of $X$.
A {\it fixed} or {\it stationary  point} of a flow is one
whose orbit, $f({\Bbb R}\times \{x\})$,  reduces to a point.
Usually, flows without fixed point are referred to as {\it
non-singular} or {\it non-stationary}.
Yet it is convenient to compactify the jargon (we follow
essentially Beck \cite{Beck_1958}, modulo a slight
compression):

\begin{defn}\label{Beck_brush:def} {\rm
%
A flow with no stationary points is called a \emph{brush}; and
a space
with a  brush is
a \emph{brushing}.
%
%
}\end{defn}

Given a flow one may consider the partition into orbits, and
expect---if both the space and the flow are sufficiently
regular (say $X$ a manifold and $f$ a brush)---a sort of
locally well-behaved geometric structure. This idea blossomed
first to the notion of {\it Kurvenschar} or {\it regular
family of curves} as defined by \Kerekjarto
\cite{Kerekjarto_1923}, Kneser \cite{Kneser24} and
Whitney~\cite{Whitney33}, and later
to the concept of {\it foliation} of Ehresmann-Reeb
(1944--1952). Thus, naively one would expect that the
partition into orbits of a brush on a manifold produces a
foliation. As we shall recall, this albeit
correct in low-dimensions $\le 3$, fails from dimension $4$,
upwards.

The other way around, given a one-dimensional orientable
foliation one may ask for a compatible flow, whose orbits
structure
generates the given foliation.
Though
non-canonical this reverse
procedure
works in full generality (modulo the metrisability axiom).

Summing up, the situation is as follows:

(1) the canonical map from
brushes to foliations is
foiled in high-dimensions $\ge 4$ (Bing-Chewning
\cite{Chewning_1974}), but well-defined in low-dimensions $\le
3$ (Whitney 1938 \cite{Whitney_1938}), and this even in the
non-metric
case (chart orbit trick).

(2) Vice versa, the non-canonical map from (oriented)
$1$-foliations to
brushes works
unconditionally in the metric realm (Whitney 1933
\cite{Whitney33}), but fails outside
(Propositions~\ref{Whitney's_flow-failure-on-plane_pipe} or
\ref{Whitney's_flow_disruption_on_Moore}).

\subsection{The foliation
induced by a
brush: Whitney-Bebutov theory}

Apart from detail of phraseology and a
non-metric shift, the following is due to Whitney:

\begin{theorem}\label{folklore} Let $f\colon{\Bbb R}\times M \to M$ be a
non-singular $C^0$-flow on a (non-metric) manifold.
Then the orbits of the flow induce a one-dimensional
oriented foliation on $M$, provided {\rm (i)} the flow is
$C^1$ or {\rm (ii)} the manifold is of dimension $n\le 3$.
\end{theorem}

\begin{proof}
In the metric case the Whitney-Bebutov theory\footnote{See
Whitney \cite[p.\,260 and 270]{Whitney33} and Bebutov 1940,
Nemytskii-Stepanov~\cite[p.\,333]{Nemytskii-Stepanov}.}
ensures the existence, through any non-singular point of the
flow,
of a local cross-section and an associated flow-box.
The metric proviso is
in fact immaterial as choosing a Euclidean chart $U$ around
any point, the chart orbit $f({\Bbb R}\times {U})$
is invariant and Lindel\"of (hence metric). 
A foliated structure follows
if one can
establish the locally Euclidean character of the
cross-section.
Whitney 1933 \cite{Whitney33} answers this
question for $n=2$ by
quoting a result of Hausdorff, while the case $n=3$ is treated
in Whitney 1938~\cite{Whitney_1938}
via his 1932
characterisation of the $2$-cell.
\end{proof}

\begin{rem}\label{warning}
{\rm
Whitney~\cite[p.\,259-260]{Whitney33} asked in 1933: {\it
given a regular family of curves
filling a region in ${\Bbb R}^n$ is there a
cross-section through any point which is a closed
$(n-1)$-cell?}  In a
related vein, O.~H\'ajek asked (1968) at the end of his
book~\cite[Problem 8, p.\,225]{Hajek_1968}
(almost verbatim): {\it Decide whether or not every continuous
dynamical system on a differential manifold is
isomorphic to a differential system.} In 1974, Chewning
\cite{Chewning_1974} provided the negative answer by
constructing a flow on ${\Bbb R}^4$ induced from a {\it
non-manifold factor} of ${\Bbb R}^4$, i.e. a non-locally
Euclidean space $X$ which crossed by ${\Bbb R}$ becomes
${\Bbb R}^4$. (Such spaces,
discovered about 1958 by Bing-Shapiro,
arise  by collapsing to a point a wild arc of ${\Bbb R}^3$.)
Chewning's negative solution to H\'ajek's problem also answers
the
1933 question of Whitney.
Regarding {\it compact} transformation groups,
non-smoothable actions were detected
earlier  in Bing~1952~\cite{Bing_1952}, showing an exotic
involution
on ${\Bbb S}^3$
by identifying the doubled
Alexander solid horned sphere
to ${\Bbb S}^3$.
(In both cases the relative dimension of the action is $3$.)

}\end{rem}

As in low dimensions $\le 3$ the relation from
brushes to foliations is
safe,
we may deduce:

\begin{cor}\label{longplane-no-brush} The long plane ${\Bbb L}^2$
lacks non-singular flows.
\end{cor}

\begin{proof} Otherwise it would have a foliation by
short leaves, violating the classification given in
\cite[Corollary~7.7]{BGG}. An alternative
(non-foliated) proof
follows either
from Theorem~\ref{omega-bounded_Hairy-ball-thm}
(Poincar\'e-Bendixson approach) or from
Corollary~\ref{fppf-Ln} (Brouwer's fixed-point theorem), which
establishes the general case of ${\Bbb L}^n$.
\end{proof}

\subsection{Whitney's flows: creating motions compatible with
a foliation
}\label{Whitney's_flow-theorem}

The
shortest route from foliations
back to dynamics is the following result of
Whitney~\cite{Whitney33}, whose $2$-dimensional case goes back
to \Kerekjarto\!\!~\cite{Kerekjarto_1925}. This
is as follows, again apart from
matters of phraseology (i.e., regular families of curves
versus foliations):

\begin{theorem}\label{Kerek_Whitney's_flow:thm} (\Kerekjarto 1925, Whitney 1933)
Given an orientable one-foliation on a metric
$C^0$-mani\-fold, there is a compatible flow whose orbits
are the leaves of the foliation.
\end{theorem}

\begin{proof} For the case of the plane, see
\Kerekjarto\cite[p.\,111, \S 7]{Kerekjarto_1925}: looking at details it seems fair to say that he
establishes the theorem in the $2$-dimensional case
(eventually with
some assistance
of Rad\'o~\cite{Rado_1925} to triangulate the surface).
%
%
The general case
is a 24 pages long
ascension \`a la Whitney \cite[Thm~27.A, p.\,269]{Whitney33}.
(For more recent treatments
compare eventually Mather~\cite{Mather82} (surfaces) and
Hector-Hirsch~\cite{HectorHirschB} (triangulations).)
\end{proof}

In passing,
we mention that, albeit limited to metric manifolds, this
result of Whitney plays a crucial r\^ole in our
previous classification of foliations on the long plane, for
the asymptotic rigidity
enables
a reduction to certain compact subregions (see
\cite[Section~7]{BGG} for the details).

\subsection{Non-metric disruption of Whitney's flows
(caused by Cantorian rigidity)}\label{Whitney's_flow}

%
%
Of course Whitney's theorem
(\ref{Kerek_Whitney's_flow:thm}) is trivially false in full
generality, because as soon as there is a long leaf (e.g., the
horizontal foliation of the long plane by long lines), the
foliation cannot be the
phase-portrait of a flow. 

Thus, the more
subtle question is whether Whitney's flows exist
when all leaves are {\it short} (i.e., metric). This fails
even when all leaves are compact (so circles) as shown by the
following example. (Later we shall assist to
another elementary disruption of Whitney's flows
on the Moore surface,
see Proposition~\ref{Whitney's_flow_disruption_on_Moore}.)

\begin{prop}\label{Whitney's_flow-failure-on-plane_pipe}
The
orientable foliation on ${\Bbb L}^2-\{0\}$
by concentric squares
lacks a compatible flow.
\end{prop}

\begin{proof}
By contradiction, let $f\colon {\Bbb R}\times M \to M$ be such
a flow, where $M={\Bbb L}^2-\{0\}$. For any $\alpha\in {\Bbb
L}_+$, let $S_{\alpha}$ be the square of radius $\alpha$ for
the ``long norm'', i.e.  $S_{\alpha}=\| \cdot
\|^{-1}(\alpha)$, where $\| \cdot \|\colon M \to {\Bbb L}_+$
is defined by $\| (x,y)\|=\max\{\vert x\vert,\vert y\vert\}$.
(The symbol $\vert \cdot\vert\colon {\Bbb L} \to {\Bbb L}_{\ge
0}$ is the obvious ``long absolute value''.) We define a map $\tau \colon
{\Bbb L}_+ \to {\Bbb R}$ taking each $\alpha \in {\Bbb L}_+$
to the time $\tau (\alpha)$ elapsed until the point $(\alpha,
0)\in {\Bbb L}^2-\{ 0\}$ returns to its initial position under
the flow $f$ (continuity is
easy to check).
%
%
%
%
Hence $\tau$ must be {\it eventually constant}, say constant
after some bound $\beta\in {\Bbb L}_+$ (see
e.g., \cite[Lemma~4.3]{BGG}). The ultimate
stagnation of the period implies that the flow ultimately
converts into an action of the circle ${\Bbb S}^1={\Bbb R}/
{\Bbb Z}$ (assuming for simplicity $\tau (\beta)=1$). More
precisely, if $M_{\ge \alpha}=\| \cdot \|^{-1} ([\alpha,
\omega_1))$ denotes the part of $M$ lying outside the square
$S_{\alpha}$, then the restricted action of ${\Bbb R}$ on
$M_{\ge \beta}$ descends to an action of ${\Bbb S^1}$ on
$M_{\ge \beta}$.

This action admits a ``slice'', $\Sigma=[\beta,\omega_1)\times
\{0\}$, 
i.e., the restricted action $\psi\colon {\Bbb S}^1 \times
\Sigma \to M_{\ge \beta}$ is a continuous bijection. 
For each $\gamma \in \omega_1$, $\psi$ restricts to a
homeomorphism $\psi_{\gamma}\colon {\Bbb S}^1 \times ([\beta,
\gamma]\times \{0 \}) \to \| \cdot \|^{-1} ([\beta, \gamma])$
between each sublevel of the canonical $\omega_1$-exhaustion
by compact annuli, hence the inverse map $\psi^{-1}$ is
continuous as well. (Alternatively one can deduce that $\psi$
is a homeomorphism from Lemma~\ref{sequentially-compact}.)
This is impossible, for
$\psi $ relates two
long pipes
belonging to distinct topological types
(cf. Lemma~\ref{planar_vs_cylindrical}).
\end{proof}

\begin{lemma}\label{planar_vs_cylindrical}
The cylindric pipe ${\Bbb S}^1 \times {\Bbb L}_{\ge 0}$ is not
homeomorphic to the planar pipe ${\Bbb L}^2-(-1,1)^2=:\Pi$.
\end{lemma}

\begin{proof} Specialists are
certainly able
to distinguish them by playing with
embedded long
rays (for a brief sketch see Nyikos~\cite[p.\,670]{Nyikos84}).
We find it however psychologically more relaxing, to
argue in terms of foliated structures; by noticing that the
planar pipe $\Pi$ lacks a foliation by long rays transverse to
the boundary. The proof (of this last statement) uses the
methods in \cite[Section~7]{BGG}; we briefly recall the idea.
By rigidity each ``rhombic'' quadrant (i.e., the results from
cuts practiced along the four diagonals of $\Pi$) is either
asymptotically foliated by long straight rays or by short
segments (cf. Figure~13 in \cite{BGG}),
giving a short list of $6$ combinatorially distinct patterns.
By suitable cuts,
we may extract 6 different compact
subregions (cf. Figure~14 in \cite{BGG}).
A plumbing argument (cf. again Figure~14) shows
 that all those $6$ patterns (except
the first) are actually impossible in view of the Euler
obstruction). The
first case left over is just a square whose boundary is a
circle leaf,
 impeding the existence of a foliation of the
specified type on $\Pi$.
\end{proof}

\begin{lemma}\label{sequentially-compact} Let $f\colon X \to Y$ be a continuous bijection. Assume
that $X$ is sequentially-compact and that $Y$ is first
countable and Hausdorff. Then $f$ is a homeomorphism.
\end{lemma}

\begin{proof} We check the continuity
of the inverse map $f^{-1}$ by showing that $f$ is closed. Let
$F$ be closed in $X$. By first countability of $Y$ it is
enough to check that $f(F)$ is sequentially closed. Let $y_n$
be a sequence in $f(F)$ converging to $y$. Let $x_n$ be the
unique lift of $y_n$ in $F$. By sequential-compactness there
is a converging subsequence $x_{n_k}$ converging to $x\in F$,
say. By continuity, it follows that $(y_n)$ converges to
$f(x)$.
By uniqueness of the limit in Hausdorff spaces, we have
$y=f(x)$, hence $y\in f(F)$, as desired.
\end{proof}

\section{Flows via small propagation}\label{Flows:Section}

By a {\it flow}, one understands a continuous group action
$f\colon {\Bbb R}\times X \to X$, of the real line on a
certain topological space. For the sake of geometric
intuition, we shall primarily deal with the case where $X$ is
a (topological) manifold,
{\it a priori} without imposing differentiability, nor
metrisability. 
%
The driving
idea in this section
is to look how the simplest available point-sets in a
manifold, namely {\it charts},
get sidetracked by the flow motion. Of special interest is the
situation, where a chart (more generally a Lindel\"of subset)
has an orbit
spreading all around the manifold (filling it completely) for
in this case the {\it whole} manifold
turns out to be
Lindel\"of, hence metric. This motivates the following jargon:

\begin{defn} {\rm A {\it propagator} for a
flow $f\colon {\Bbb R} \times X\to X$ is a subset $\Sigma$ of
the phase-space, $X$, such that the restricted map of the flow
$f\colon {\Bbb R} \times \Sigma \twoheadrightarrow X$ is
surjective. The propagator is said to be {\it small} if it is
Lindel\"of. In that case the phase-space $X$
is Lindel\"of (for
${\Bbb R}\times \Sigma$ is Lindel\"of, as ${\Bbb R}$ is
$\sigma$-compact).}
\end{defn}

\subsection{Extinction of  minimal flows on non-metric
manifolds}\label{minimal-inexistence}

The following
is a very baby non-metric version of the Gottschalk conjecture
\cite{Gottschalk_1958}, inquiring which spaces, especially
manifolds,
carries a minimal flow:

\begin{prop} \label{minimal-inexistence-prop}
A manifold carrying a minimal flow is metric.
\end{prop}

\begin{proof}
In a minimal flow, each non-empty open set is a propagator.
\end{proof}

\begin{rem}\label{cytoplasmic_expansions}
(Minimal foliations via Nyikos' cytoplasmic expansions) {\rm
In contrast, it
is possible to construct  minimal foliations on non-metric
surfaces,
via a Kronecker (irrational slope) foliation on the $2$-torus.
Recall a remarkable construction of Nyikos \cite{Nyikos90},
attaching a long ray to
certain
surfaces. The plane ${\Bbb R}^2$, for instance, can
be stretched in one or more directions to give ``amoebas''
with long {\it cytoplasmic expansion(s)} of the $2$-cell (by
so-called
{\it pseudopodia}). When applied to a punctured Kronecker
torus, such a long cytoplasmic expansion---effected near the
puncture
in a way parallel to the foliation---produces a minimally
foliated non-metric surface
with one long-ray leaf. Besides, by a clever Pr\"uferisation
along a suitable Cantor set
in the torus, M. Baillif
also proposes a minimal foliation having only metric leaves
(details may appear in a subsequent paper Baillif~{\it et
al.}~\cite{BGG2}).}
\end{rem}

\subsection{Disruption of Whitney's flows detected by small
propagation}

The idea of
propagation
is now used to identify more examples where Whitney's flows
run into troubles, e.g. on the {\it Moore surface}.
%
%
Recall, the latter
to be deduced from the bordered Pr\"ufer surface, $P$, by
self-identifying the boundary components via the folding
$x\sim -x$.
(For the original description cf.
\cite{Moore_1942} and the 1962 edition of the {\it
Foundations} \cite[p.\,376--377]{Moore_1962}.)
The horizontal foliation on $P$ when pushed down to $M$
develops many {\it thorns} singularities (cf.
Figure~\ref{Artist_views}). Thus, we  rather consider the
vertical ``foliation'' on $P$
$(x={\rm const})$, which in fact is {\it not}  a genuine
foliation,  for it has {\it saddle}
singularities (locally equivalent to the level curves of the
function $xy$, restricted to the upper-half plane).
Yet, the quotient mapping
$P \to M$ resolves these
singularities to induce a (genuine) foliation on the Moore
surface $M$, which we refer to as the {\it vertical foliation}
of $M$.

\begin{prop}\label{Whitney's_flow_disruption_on_Moore}
The vertical foliation on the Moore surface, $M$,
lacks a compatible Whitney flow.
\end{prop}

\begin{proof}  Notice that the
horizontal line $y=1$ is a small propagator for any compatible
flow, violating the non-metrisability of the Moore surface.
\end{proof}

This exemplifies probably one the simplest manifestation of
the radical
divorce affecting (outside the metric sphere)
the fusional cohesion  of ``flows'' with ``foliations'',
under Whitney's marriages.

\subsection{Pr\"ufer's flows: windscreen wiper motions}

Small propagators are not always available, e.g., in the case
of the horizontal foliations on the Pr\"ufer surfaces (either
$2P$ or $P_{\rm collar}$).
The natural candidates would be
countable unions of vertical lines, but then some of the
uncountably many {\it bridges} (cf. Figure~\ref{Artist_views},
for an intuitive meaning) remain unexplored .
Thus, the ``flowability'' of these foliations is not
obstructed by
small propagation,
and it turns out that those horizontal foliations on the
varied Pr\"ufer
surfaces admit indeed a Whitney flow. Heuristically, such a
flow (on $P$) can be visualised as the motion of a
``windscreen wiper'', where the different rays (involved in
Pr\"ufer) are undergoing a collective sweeping motion
(providing thereby a first interesting motion (brush) on a
non-metric manifold, not merely
induced by a metric factor). Here is a formal treatment:

\begin{prop} \label{Pruefer_flow} The horizontal
foliation on the bordered Pr\"ufer surface $P$ admits a
compatible Whitney flow.
\end{prop}

\begin{proof} 
%
Such a flow $f\colon {\Bbb R}\times P\to P$
is obtained as follows: given $p\in P$ and $t\in{\Bbb R}$
there are two cases:

(i) If $p$ is a ``genuine'' point (i.e., not a ``ray'') then
$p=(x,y)$, and define $f(t, (x,y))=(x+yt,y)$. [Geometrically,
 given $p$,  draw the vertical line through $p$
and take its intersection with the ($y=1$)-line, to which a
horizontal translation of amplitude $t$ is operated (to get
the point $q=(x+t,1)$), and  $f(t,p)$ is the intersection
point of the line through $(x,0)$ and $q$ with the horizontal
line at height $y$.]

(ii) If $p$ is a ray emanating from $(x,0)$, then $p$ has some
``slope'' $s$ defined as $\frac{\Delta x}{\Delta y} $, and
$f(t,p)$ is defined as the ray through $(x,0)$
of slope $s+t$.

Both prescriptions are consistent: consider a
sequence
$p_i$ of points converging to a ray through $(x_0,0)$ of slope
$s$. The ray has an equation $x=sy+x_0$, and we may assume
that the $p_i=(x_i,y_i)$ lye (eventually) on this ray while
converging to it. Then
$f(t,p_i)=(x_i+y_it,y_i)=(sy_i+x_0+y_it,y_i)=((s+t)y_i+x_0,y_i)$,
showing that the $f(t,p_i)$'s  converge to the ray of slope
$s+t$ through $(x_0,0)$.

Finally the group property
holds:  $f(t',f(t,(x,y)))=f(t',
(x+yt,y))=(x+yt+yt',y)=(x+y(t+t'),y)=f(t+t',(x,y))$, while for
rays the property is obvious.
\end{proof}

\begin{exam}\label{Kronecker-Pruefer} {\rm Such Pr\"ufer
flows
also exist on the $2$-torus irrationally foliated and
Pr\"uferised along a closed interval
in a leaf. (This is made more explicit in case (16) of
Section~\ref{Geography}.) If the Pr\"uferisation
is effected so
that it ``locally'' resembles the doubled Pr\"ufer $2P$
(leading to a separable surface), then the Pr\"ufer flow
is {\it transitive} (i.e., exhibit at least one dense orbit).
In fact, only the points
lying on the ``bridges''
lack a dense orbit. }
\end{exam}

What is the minimal cardinality of
non-dense orbits for a flow on a non-metric manifold? Small
propagation again gives a quick (Cantor relativistic) answer:

\begin{prop}\label{pseudo-minimal}
A flow on a non-metric manifold has uncountably many non-dense
orbits.
\end{prop}

\begin{proof} Otherwise a small propagator is designed by
aggregating to a chart countably many points picked in the
non-dense orbits.
\end{proof}

Thus, the minimal cardinality
in question is at least $ \omega_1$ (and at most $\frak c=
{\rm card} {\Bbb R}$, for the cardinality of a connected
Hausdorff $n$-manifold is $\frak c$, provided $n\ge 1$
\cite{Spivak}, \cite[Thm 2.9]{Nyikos84}). The Pr\"ufer flow
of Example~\ref{Kronecker-Pruefer} has exactly $\frak c$ many
non-dense orbits, hence
realises the lower-bound
under the continuum hypothesis (CH). If the negation of (CH)
holds (i.e., $\omega_1 < \frak c$), we may
throw away  non-dense orbits until precisely $\omega_1$ are
left. In conclusion the answer is $\omega_1$, yet the
``exact'' size of $\omega_1$ depends on (CH).

\subsection{Cytoplasmic obstruction to non-singular flows}\label{CTD:Section}

Since non-metric manifolds
lack the best possible {\it minimal} dynamics, we switch to
the weaker paradigm of non-singular flows (alias {\it
brushes}), by
wondering which
manifolds can support them?

From
(\ref{Pruefer_flow}),
the Pr\"ufer surfaces (in all
its three
incarnations: $P$, $2P$ and $P_{\rm collar}$) all admit a
brush. By contrast, we know already one
non-metric surface lacking a brush, namely the long plane
${\Bbb L}^2$
(\ref{longplane-no-brush}).
Regarding the Moore surface, we noticed in
(\ref{Whitney's_flow_disruption_on_Moore}) that the natural
vertical foliation
lacks a compatible brush (a priori not excluding the
existence of a brush
inducing a more exotic foliation).
Yet, no such exotica with  Moore:
\begin{prop}
The Moore surface lacks any brush.
\end{prop}

\begin{proof}
The Moore surface consists of a ``core'' corresponding to the
interior of $P$, plus a continuous collection of thin
``thorns'' hanging on (picturesquely like thin
stalactites), cf. Figure~\ref{Artist_views}. This very
specific
morphology implies the core to be a propagator under any
brush, for any point on a thorn (homeomorphic to a semi-line
${\Bbb R}_{\ge 0}$) will eventually flow into the core.
\end{proof}

This argument readily generalises to the following class of
surfaces (indeed manifolds of arbitrary dimensionality)
abstracting
the morphology of the Moore surface. (This structure---which
we
shall refer to as {\it cytoplasmic}---may perhaps be regarded
as a ``separable'' counterpart
to
the {\it bagpipe} structure of Nyikos~\cite{Nyikos84}
concretising the point-set paradigm of $\omega$-boundedness):

\begin{defn} {\rm A
manifold,
$M$, is said to admit a {\it cytoplasmic} (or {\it
core-thorn})
{\it decomposition} {\rm (both abridged CTD)} if it contains
an open set $U$ which is Lindel\"of (the {\it core}) such that
the residual set $M-U$ decomposes into connected components as
$\bigsqcup_{x\in X} T_x$, where each {\it thorn} $T_x$ is a
bordered one-manifold with a single boundary-point, so
homeomorphic either to ${\Bbb R}_{\ge 0}$ or to ${\Bbb L}_{\ge
0}$.
(The unique boundary point of each thorn is used as indexing
parameter.) Further we assume the following axiom:

{\rm (CT)} Each point $z\in T_x$ on a thorn admits a
fundamental system of open neighbourhoods $U_z$ consisting of
points on the thorn plus some non-empty set of point in the
core.}
\end{defn}

The following properties are easily verified:

(CT1) For each thorn $T_x$, the set $U\cup T_x$ is
open.

(CT2) The core $U$ is dense in $M$, and being Lindel\"of
(hence separable, since manifolds are locally second
countable) it follows that $M$ is separable.

\smallskip
Surfaces
tolerating
a cytoplasmic decomposition include the Moore and the {\it
Maungakiekie surface\footnote{So called after a certain hill
in New Zealand surmounted by a thin tower.}}. The latter
refers to the result of a unique cytoplasmic expansion of an
open $2$-cell by a long ray, as discussed in
(\ref{cytoplasmic_expansions}). To mention an example in
dimension $3$, one may consider the $3$-dimensional avatar of
the bordered Pr\"ufer surface, say $P^3$ (likewise constructed
from the half-$3$-space by adjunction of an ideal boundary of
``rays'') and fold the boundary-components (homeomorphic to
${\Bbb R}^2$) by collapsing all points on concentric circles.

\begin{exam} \label{Pruefer_Moore:fabric:exam}
(Pr\"uferisation and Moorisation) {\rm For more  (including
non simply-connected) examples, one needs only to
plagiarise the Pr\"ufer construction
over any
metric bordered surface, $W$, (e.g., an annulus ${\Bbb
S}^1\times[0,1]$), producing a non-metric bordered surface,
$P(W)$, called the {\it Pr\"uferisation} of $W$. Then, folding
the boundaries gives a {\it Moorisation}, $M(W)$. The latter
is separable and comes with a (canonic) cytoplasmic
decomposition. (For more details, compare
Definition~\ref{Pruefisation-Moorisation}.)}
\end{exam}


\begin{theorem} \label{CTD:thm} A non-metric
manifold with a
cytoplasmic decomposition $M=U \sqcup \bigsqcup_{x} T_x$ has
no brush.
\end{theorem}

\begin{proof} It is again the matter of observing that the core
is a small propagator for any brush on $M$.
\end{proof}

Corollaries of the theorem includes:

\begin{cor}
{\rm (i)} The Moore surface (and all its avatar $M(W)$ in {\rm
(\ref{Pruefer_Moore:fabric:exam})}), as well as the
Maungakiekie (and a myriad of other
``protozoans'' deduced via cytoplasmic expansions) are not
brushings. 

\noindent {\rm (ii)} The doubled Pr\"ufer surface, $2P$, as it
is a brushing {\rm (\ref{Pruefer_flow})}, does not
tolerate a cytoplasmic decomposition.

\end{cor}

The failure of the Moore surface to brush suggests the
following question (which will be answered as
Theorem~\ref{hairy_ball_thm:separable}):

\begin{ques}\label{Pseudo-Moore:conj:question} (Pseudo-Moore problem) A pseudo-Moore
surface---defined as one  satisfying the following four axioms
(verified by the classic Moore surface): {\rm (i)}
simply-connected, {\rm (ii)} separable, {\rm (iii)}
non-metric, {\rm (iv)} without boundary---lacks a brush.
\end{ques}

Relaxing any of the four assumptions
foils the conclusion: If (i) is relaxed take $2P$, if (ii) is
omitted ${\Bbb R} \times {\Bbb L}$ works, if (iii) is
suppressed ${\Bbb R}^2$ is fine, and if (iv) is left $P$ is
suitable. The two-dimensionality is also crucial, since the
$3$-manifold ${\Bbb R} \times ({\rm Moore})$ has a brush.

Our initial motivation for
cytoplasmic decompositions
was an attempt to solve the above Pseudo-Moore problem.
(Speculating that any pseudo-Moore surface permits a CTD,  the
problem reduces to Theorem~\ref{CTD:thm}.) Our later solution
follows an entirely different route---at least in
appearances---using more
geometric methods (non-metric
Schoenflies and phagocytosis). 
%
%
Yet, by analogy with the ubiquity of Nyikos' bagpipe structure
(in the $2$-dimensional realm at least), it sounds natural to
conjecture:

\begin{conj}\label{cytoplasmic-dec:conj} Any
separable $2$-manifold with
countable fundamental group
has a cytoplasmic decomposition.
\end{conj}

The proviso on the fundamental group is of course essential
(else $2P$ is a counterexample). The more intrinsic reason is
of course that the condition
looks necessary: for, given a cytoplasmic structure on a
surface $M$, it is likely that the natural morphism $\pi_1(U)
\to \pi_1(M)$ induced by the ``core'' inclusion is isomorphic
(we leave this in {\it une ombre propice}\footnote{Well-known
phraseology of Ren\'e Thom.}, since we shall not use it). At
any rate, a dynamical consequence of
(\ref{cytoplasmic-dec:conj}) is that any such surface, if
non-metric, lacks a brush in view of (\ref{CTD:thm}). This
assertion will be proved later
(Corollary~\ref{separable_hairy_ball_pi_1_ctble:cor}), via
a trivial reduction to the simply-connected case
(passage to the universal covering). In case
Conjecture~\ref{cytoplasmic-dec:conj} should be
 true (and of real independent
interest), then it is quite likely that the proof will depend
on techniques \`a la Morton Brown (cf. especially what we call
the {\it inflation conjecture} discussed later as
item~(\ref{Inflation:conj}): this is susceptible to produce
the ``core'' via inflation of an open set covering faithfully
a basis of the $\pi_1$, after using Schoenflies to fill
inessential circles).


\subsection{Dynamical Euler characteristic:
Cartesian multiplicativity? }

Now, a little curiosity: define the {\it (dynamical)
characteristic} $\delta(M)\in {\Bbb F}_2$ (the field with two
elements) of a manifold $M$ as equal to $0$ if $M$
has a brush, and $1$ otherwise.

\begin{ques}\label{multiplicative} Is $\delta$
multiplicative under Cartesian products, i.e. $\delta(M\times
N)=\delta(M)\cdot \delta (N)$?
\end{ques}

The formula is obvious if one of the two factors has vanishing
$\delta$-characteristic (one can brush the product from one of
its factor). If ``multiplicativity'' holds it would
imply $\delta({\Bbb L}^n)=1$ (and
$\delta({\Bbb
L}_{+}^n)=1$). This
reminds the question of Kuratowski as to whether the
product of two spaces having the fpp (fixed point property)
still has the fpp: the failure is well-known (cf. R.\,F.
Brown's survey~\cite{Brown_RF_1982}).
The case $n=2$ of $\delta({\Bbb L}^n)=1$ follows from the
classification of foliations on ${\Bbb L}^2$
(\ref{longplane-no-brush}). Maybe, the square of the Moore
surface, ${M}\times M$,  provides a negative answer to
(\ref{multiplicative}): for ${ M}\times { M}$ may have a
brush, since the product of thorns $T_x$ are quadrants ${\Bbb
R}_{\ge 0}^2$ which tolerate a brush
(not constrained
to move into the core).

\subsection{Fixed point property for flows via
sequential-compactness: the case of ${\Bbb L}^n$}

The general case of $\delta({\Bbb L}^n)=1$ can be proven
independently of (\ref{multiplicative}) via the:

\begin{prop}\label{fpp-for-flows} Let $M$ be a
sequentially-compact manifold with the fpp (fixed point
property).  Then $\delta(M)=1$, i.e. $M$ has the fpp for
flows. (In fact ``manifold'' can be replaced by ``Hausdorff
space'', and it is enough to assume fpp for homeomorphisms
isotopic to the identity.)
\end{prop}

\begin{proof} This is the standard\footnote{Compare,
e.g., Lima 1964 \cite[Lemma~4,
p.\,101]{Lima_1964},
Hector-Hirsch~\cite{HectorHirschB}, Vick~\cite[Thm\,7.28,
p.\,208]{Vick}. }
``dyadic cascadisation'' argument.
Let $f\colon {\Bbb R} \times M \to M$ be any flow. Write $f_t$
for the time-$t$ map, i.e. $f_t(x)=f(t,x)$. Let $t_n=1/2^n$
for $n$ an integer $\ge 0$ and let $K_n$ be the fixed-point
set of $f_{t_n}$. The sets $K_n$ are closed ($M$ is
Hausdorff), non-empty ($M$ has the fpp) and  nested
$K_n\supset K_{n+1}$ (as follows from the group property).
Next, observe the:

\begin{lemma}\label{nested_intersection}
If $(K_n)_{n\in \omega}$ is a nested sequence of non-empty
closed sets in a sequentially-compact space $X$, then the
infinite intersection $\bigcap_{n\in \omega} K_n$ is
non-empty.
\end{lemma}

\begin{proof} Choose, for each $n\ge 0$, a point $x_n \in K_n$.
By sequential-compactness the sequence $(x_n)$
has a converging subsequence, whose limit $x$
belongs to the intersection of all $K_n$'s.
\end{proof}

\noindent By  (\ref{nested_intersection}) there
is a point in $\bigcap_{n\ge 0} K_n$, thus fixed under all
dyadic times,
hence under all
times.
\end{proof}

It remains to
notice
a transfinite avatar of the Brouwer fixed point theorem:

\begin{lemma}\label{transfinite-Brouwer}
Let $\mathbb{L}$ be the long line, then
$\mathbb{L}^n$ has the
fixed point property for all $n$.
\end{lemma}

\begin{proof}
(Communicated by Mathieu Baillif.)
  Let $f\colon \mathbb{L}^n\to\mathbb{L}^n$ be continuous, and
  let $U_\alpha$ be $(-\alpha,\alpha)^n\subset\mathbb{L}^n$ for
  $\alpha\in\omega_1$. Since $U_\alpha$ is Lindel\"of, $f(U_\alpha)$
  is contained in $U_{\beta(\alpha)}$ for some
  $\beta(\alpha)\ge \alpha$. Set
  $\alpha_0 = 1$ and $\alpha_n=\beta(\alpha_{n-1})$ for $n\in\omega$.
  Let $\alpha$ be the supremum of the $\alpha_n$'s, then
  $f(U_\alpha)\subset U_\alpha$. By continuity,
  $f(\overline{U_{\alpha}})\subset \overline{f(U_\alpha)}
  \subset \overline{U_\alpha}$, where
  $\overline{U_\alpha}=[-\alpha,\alpha]^n$ is a topological
  ball, and we conclude via
  the Brouwer fixed point theorem.
\end{proof}

\begin{cor}\label{fppf-Ln}
The long hyperspace $\mathbb{L}^n$ has the fixed-point
property for flows for all integer $n\ge 0$.
\end{cor}

This method
has
a
wider flexibility, for it applies to
other (balls-exhaustible) manifolds
and
one  may also harpoon the fixed-point via Lefschetz, instead
of Brouwer (these ideas
are discussed in
Section~\ref{high-dim-hairy-ball-thms}).

\section{Dynamical consequences of Schoenflies}

We now address
certain
dynamical
repercussions of the ``Schoenflies axiom'' (any embedded
circle bounds a disc). In view of \cite{GaGa2010}, this
amounts to simple-connectivity. This section certainly
contains the most satisfactory results of the paper, inherited
from the
``tameness'' of $2$-dimensional---and the allied
plane---topology. This explains why the classical paradigms of
$2$-dimensional dynamics transpose so easily to the non-metric
realm (without having to control the growing mode of the
manifold). As we shall oft deplore later, nothing similar
seems to occur in higher-dimensions
(and to be perfectly honest, not even in $2$-dimensions when
the fundamental group is allowed to be non-trivial---recall
the plague of wild pipes).

\subsection{Poincar\'e-Bendixson theory
(dynamical consequence of Jordan)}

Our first objective is
Theorem~\ref{omega-bounded_Hairy-ball-thm} below, whose formal
proof requires some preparatory lemmata, all very standard.
For the
sake of
short-circuiting some logical deductions
by geometric intuition here is a quick justification of
(\ref{omega-bounded_Hairy-ball-thm}):
by $\omega$-boundedness the surface is ``long'' impeding a
``short'' semi-orbit to escape at infinity. Thus, any
trajectory
begins a spiraling motion,
creating asymptotically either a stationary point or
a periodic orbit. In the latter case, Schoenflies gives an
invariant bounding disc, where  a fixed point is created by
Brouwer. {\it q.e.d.}

The rest of this subsection
details the more pedestrian
route (thus skip it, if you like):

\begin{lemma}\label{Brouwer} A flow
on a Schoenflies surface with a periodic orbit has a
fixed-point. (Same conclusion if the periodic orbit is
null-homotopic.)
\end{lemma}

\begin{proof} The periodic orbit is a topological circle,
so bounds a 2-disc, which is invariant under the flow.
Applying the Brouwer fixed point theorem to the dyadic times
maps $f_{t_n}$ of the flow ($t_n={1}/{2^n}$), we get a nested
sequence of non-empty fixed-point sets $K_n={\rm
Fix}(f_{t_n})$ whose common intersection $\bigcap_{n=
1}^{\infty} K_n$ is non-empty  (compactness of the disc). By
continuity, a point in this set
is fixed under the flow.
\end{proof}

Next, we seek for an intrinsic topological property forcing
the formation of periodic orbits.
A good condition is $\omega$-boundedness\footnote{Recall that
a space is {\it $\omega$-bounded} if each countable subset has
a compact closure.}: for given a flow $f$, the set ${\Bbb Q}
x=f({\Bbb Q} \times \{x\})$ is countable, hence its closure
$\overline{{\Bbb Q} x}$ is compact. Note that $\overline{{\Bbb
Q} x}=\overline{{\Bbb R} x}$ (as follows from the general
formula $f(\overline S)\subset \overline {f(S)}$ for a
continuous map $f$).
A periodic motion
is produced by the following standard mechanism
(of G.\,D. Birkhoff):

\begin{lemma}\label{Birkhoff} Let $f$ be a
flow on a (Hausdorff) space $X$. Assume that the orbit closure
$\overline{{\Bbb R}x}$ is compact and
the flow
\emph{proper} (i.e. each ``inherent'' open set of the form
$f(]s,t[\times \{z\})$ is also open for the relative topology
on ${\Bbb R}z$). Then
there is a compact orbit, which by properness is either a
fixed point or a periodic orbit.
\end{lemma}

\begin{proof}  Consider the set $\Sigma$ of all
non-empty closed $f$-invariant subsets of the compactum
$\overline{{\Bbb R}x}$. Order it by inclusion. Inductiveness
follows
from the non-emptiness of a nested intersection of closed sets
in a compactum. By Zorn's lemma, there is $K$ a minimal set in
$\Sigma$.

We show that $K$ reduces to a single {\it orbit}. Indeed let
$y\in K$, then ${\Bbb R}y\subset \overline{{\Bbb R}y}\subset
K$, and the last inclusion is an equality by minimality of
$K$. So it is enough to check that ${\Bbb R}y$ is closed. If
not then $\overline{{\Bbb R}y}-{\Bbb R}y$ is non-empty,
invariant and closed (by
Lemma~\ref{properness}
below), contradicting minimality.
It remains to check:
\begin{lemma}\label{properness} Let $f$
be a flow on a (first countable and Hausdorff) space $X$. If
the orbit ${\Bbb R}x$ is proper, then the set $\overline{{\Bbb
R}x}-{\Bbb R}x$ is closed (in $X$).
\end{lemma}

\begin{proof} It is enough to show that ${\Bbb R}x$
is open in $\overline{{\Bbb R}x}$. Let $z\in {\Bbb R}x$.
Choose any $\varepsilon>0$, then
$f(]-\varepsilon,\varepsilon[\times\{z\})$ is open for the
inherent topology on ${\Bbb R}x$, so by properness there is
$U$ open in $X$ such that $U \cap {\Bbb
R}x=f(]-\varepsilon,\varepsilon[\times\{z\})$. Then $U \cap
\overline{{\Bbb R}x}$ is open in $\overline{{\Bbb R}x}$,
contains $z$ and satisfies $U \cap \overline{{\Bbb R}x}\subset
{\Bbb R}x$.

Indeed let $y\in U \cap \overline{{\Bbb R}x}$. Since $X$ is
first countable, we choose an approximating sequence
$y_n\in{\Bbb R}x$ converging to $y$. Since $U$ is open, there
is an integer $N$ such that $y_n\in U$ for all $n\ge N$. So
$y_n \in U \cap {\Bbb R}x$, hence $y_n=f(t_n, z)$ with $\vert
t_n \vert < \varepsilon$. But since
$f([-\varepsilon,\varepsilon]\times \{z\})=:F$ is compact
(hence closed in $X$ Hausdorff), we have that $y\in F$. Since
$F \subset {\Bbb R}z={\Bbb R}x$, this completes the proof.
\end{proof}

\noindent Summing up, $K$ is a compact orbit, and properness implies that orbits are
connected Lindel\"of $n$-manifolds with $n\le 1$, so that $K$
is either a point or a circle.
\end{proof}

For surfaces,  properness is ensured by the {\it dichotomy} of
the underlying surface (i.e., Jordan separation by circles
holds true):

\begin{lemma}\label{Bendixson} Every flow on a
dichotomic surface is proper.
\end{lemma}

\begin{proof} (The classical
Bendixson
sack argument.) Given an inherent open set
$I_{\varepsilon}:=f(]-\varepsilon,\varepsilon[\times \{ x\})$
we seek an open set $U$
such that $U\cap {\Bbb
R}x=I_{\varepsilon}$. If the point $x$ is
stationary,
$U$ is easy-to-find. If non-stationary the theory of
flow-boxes is available.
In this theory, metrisability plays a crucial r\^ole, which
turns out however to be subsidiary; for letting flow any chart
domain $V$ about $x$ yields the set $f({\Bbb R}\times {V})$
which is Lindel\"of,
hence metric. According to Bebutov (cf. Nemytskii-Stepanov
\cite[p.\,333--335]{Nemytskii-Stepanov}) a flow-box can be
found for any preassigned time length, shrinking eventually
the cross-section $\Sigma_x\ni x$.
 Inside the flow-box $B:=f(  [-\varepsilon, \varepsilon] \times
\Sigma_x)$ a certain sub-rectangle $R$ is singled out by the
first returns of $x$ to the section $\Sigma_x$ (both forwardly
and backwardly in time). Dichotomy ensures via a Bendixson
sack argument the absence of further ``recurrences'', i.e.
subsequent returns intercepting the section $\Sigma_x$ closer
to $x$. The existence of the desired $U$ is guaranteed (the
interior of $R$ is appropriate).
\end{proof}

\subsection{$\omega$-bounded hairy ball theorem}

Assembling the previous facts, we
obtain:

\begin{theorem} \label{omega-bounded_Hairy-ball-thm}
Any flow on a $\omega$-bounded
Schoenflies surface
(equivalently simply-connected)  has a fixed point.
\end{theorem}

\begin{proof} Since
simply-connected implies dichotomic~\cite{GaGa2010}, the flow
is proper
by (\ref{Bendixson}). Take any point $x\in M$, its orbit
closure $\overline{{\Bbb Q}x}=\overline{{\Bbb R}x}$ is
compact, by $\omega$-boundedness. By
(\ref{Birkhoff}), the flow has either a fixed point or a
periodic orbit. In the latter case one applies
(\ref{Brouwer}).
\end{proof}

This may be regarded as a non-metric
avatar of the ``hairy ball theorem'' (Poincar\'e, Dyck,
Brouwer): the \hbox{2-sphere} cannot be foliated nor brushed.
The
theorem applies for instance to the long plane ${\Bbb L}^2$
(which case also follows from the classification of foliations
on ${\Bbb L}^2$ given in \cite{BGG}). It also applies to any
space obtained from a Nyikos long pipe, \cite{Nyikos84}, by
capping off the short end by a 2-disc, for example the {\it
long glass}, i.e. the semi-long cylinder ${\Bbb S}^1\times
\mbox{(closed long ray)}$ capped off by a 2-disc.
The proof just given also shows:

\begin{cor}
Any flow on a $\omega$-bounded dichotomic surface has either a
fixed point or a periodic orbit.
\end{cor}

This applies to any surface
deduced from the $2$-sphere by insertion of a finite number of
long pipes. (For instance any brush on the long cylinder
${\Bbb S}^1 \times {\Bbb L}$ has a periodic motion.)

\subsection{Fragmentary high-dimensional
hairy ball theorems via
Lefschetz}\label{high-dim-hairy-ball-thms}

It seems
natural to wonder if
Theorem~\ref{omega-bounded_Hairy-ball-thm} generalises to
dimension $3$ (and higher).
A basic
corruption is the (compact) 3-sphere ${\Bbb S}^3$ brushed
via a (Clifford-)Hopf fibering (${\Bbb S}^1$-action
arising from the ambient complex coordinates).
By the Poincar\'e conjecture this is the only compact
counterexample.
The following speculates
this to be the only failure in general:

\begin{conj}\label{3D-hairy-ball-thm} Any simply-connected
$\omega$-bounded non-metric $3$-manifold has the fpp for
flows.
\end{conj}

The
same
in dimension $n\ge 4$
is foiled: consider ${\Bbb S}^3 \times {\Bbb L}^{n-3}$ brushed
along the first factor by the Hopf fibering.
Let us try to collect some experimental evidence towards
(\ref{3D-hairy-ball-thm}). First,
the assertion
holds for ${\Bbb L}^3$ (Corollary~\ref{fppf-Ln}). By
the
proof of Lemma~\ref{transfinite-Brouwer}, the conjecture holds
more generally if the $3$-manifold admits an
$\omega_1$-exhaustion by (compact) $3$-balls (at least if the
interior exhaustion is {\it canonical}---in the sense of
Nyikos discussed below). However, this
 fails to cover the general case, as
${\Bbb S}^2\times {\Bbb L}$ lacks
a ball-exhaustion having a non-trivial second homotopy group
$\pi_2$.
To handle this situation we use Lefschetz
in place of Brouwer's fixed point theorem, to obtain:

\begin{prop}\label{Lefschetz-fppf} Let $M$
be an $\omega$-bounded
$n$-manifold. Assume the
existence of a canonical exhaustion
$M=\bigcup_{\alpha<\omega_1} W_\alpha$ by compact
bordered\footnote{Recall that, {\it bordered manifold} is
understood as a synonym of manifold-with-boundary. Moreover
``canonical'' refers to the hypothesis that the
``interiorised'' exhaustion $M=\bigcup_{\alpha<\omega_1} {\rm
int} (W_\alpha)$ by the interiors of the $W_{\alpha}$'s
satisfies the continuity axiom of
Lemma~\ref{exhaustion-preservation} right below.}
$n$-submanifolds $W_\alpha$ with non-vanishing
Euler characteristic. Then any flow on $M$ has a fixed point.
\end{prop}

\begin{proof}
By the cascadisation argument in the proof of
(\ref{fpp-for-flows}), it is enough to show that $M$ has the
fpp for a map $f$ homotopic to the identity (in fact,
sequential compactness follows from the exhaustion assumption,
so that we could relax ``$\omega$-boundedness'').
Lemma~\ref{exhaustion-preservation} below produces some big
$\alpha\in \omega_1$ so that $f(W_{\alpha})\subset W_\alpha$,
and with corresponding restriction $f_{\alpha}=f\vert
W_\alpha\colon W_{\alpha}\to W_{\alpha}$ homotopic to the
identity. The hypothesis $\chi(W_\alpha)\neq 0$ implies via
the Lefschetz fixed point theorem (for compact metric ANR's
\cite{Lefschetz_1937}) that $f_{\alpha}$ has a fixed point.
\end{proof}

\begin{lemma}\label{exhaustion-preservation} Let
$M=\bigcup_{\alpha<\omega_1} M_\alpha$ be a
space with an $\omega_1$-exhaustion by Lindel\"of open subsets
$M_\alpha$ verifying the continuity axiom
$M_{\lambda}=\bigcup_{\alpha<\lambda} M_{\alpha}$ whenever
$\lambda$ is a limit ordinal.

{\rm (i)} Given a continuous map $f\colon M\to M$ there is a
club\footnote{Closed unbounded subset of $\omega_1$.} $C$ of
indices $\alpha$ such that $f(M_\alpha)\subset M_\alpha$.

{\rm (ii)} Moreover if $f$ is homotopic to the identity of
$M$, then there is some $\alpha\in \omega_1$ so that the
restriction $f\colon \overline{M_\alpha} \to
\overline{M_\alpha}$ is homotopic to the identity.
\end{lemma}

\begin{proof} (i) The image $f(M_\alpha)$ is Lindel\"of, hence there
is a $\beta(\alpha)\ge\alpha$ so that $f(M_\alpha)\subset
M_{\beta(\alpha)}$. Define inductively $\alpha_1=1$,
$\alpha_n=\beta(\alpha_{n-1})$ and let $\alpha=\sup_n
\alpha_n$. By the continuity axiom
$M_{\alpha}=\bigcup_{n=1}^{\infty} M_{\alpha_n}$.
Consequently,  $f(M_\alpha)=f(\bigcup_{n=1}^{\infty}
M_{\alpha_n})=\bigcup_{n=1}^{\infty}f(M_{\alpha_n})\subset
\bigcup_{n=1}^{\infty}M_{\alpha_{n+1}}=M_{\alpha}$. This
implies the first clause. 

(ii)
Let $(f_t)$, $t\in [0,1]$ be a homotopy relating ${\rm
id}_{M}=f_0$ to $f=f_1$. Let $D$ by the set of dyadic numbers
in $[0,1]$. For each $t\in D$, the map $f_t$ preserves a club
$C_t$ of stages $M_{\alpha}$
by (i).
As any countable intersection of clubs is again a
club\footnote{This is an
exercise related to the leapfrog
argument. Hint: the leapfrog argument shows the case of a
two-fold intersection. Induction the case of a finite
intersection. The countable case intersection easily follows:
if $C_n$, $n\in \omega$ is a countable collection of clubs,
choose $x_1\in C_1$, find a
greater $x_2\in C_1\cap C_2$, and so on $x_n\in C_1\cap\dots
\cap C_n$
greater than $x_{n-1}$, then $\sup_{n<\omega} x_n$ belongs to
all $C_n$'s. {\it q.e.d.}},
$C=\bigcap_{t\in D} C_t$ is a club indexing stages
preserved by all dyadic-times of the homotopy $(f_t)$, i.e.
$f_t(M_\alpha)\subset M_\alpha$ for all $t\in D$ and all
$\alpha\in C$. By continuity $f_t(\overline{M_\alpha})\subset
\overline{f_t(M_\alpha)}\subset\overline{M_\alpha}$. By
joint-continuity, the inclusions
$f_t(\overline{M_\alpha})\subset\overline{M_\alpha}$ (for
$\alpha\in C$) hold indeed for all $t\in [0,1]$.
%
%
%
%
This
provides the required homotopy.
\end{proof}

Typical illustrations of Proposition~\ref{Lefschetz-fppf}
arise from the $3$-sphere by excising a finite number $n\ge 1$
of (tame) \hbox{$3$-balls}, and then capping off by long
3D-pipes, e.g. of the form ${\Bbb S}^2\times {\Bbb L}_{\ge
0}$. (Other pipes can be used, provided they are sufficiently
``civilised'' to allow an exhaustion of the bordered canonical
type.) Thus, all such $3$-manifolds have the fpp for flows,
provided $n\ge 1$.

\begin{rem} {\rm
By the Poincar\'e conjecture, the above family of examples is
essentially exhaustive. Indeed, {\it a simply-connected
bordered compact $3$-manifold $W$
is a holed $3$-sphere}, i.e.,  ${\Bbb S}^3$ excised by the
interior of a disjoint finite family of tame $3$-balls.
[Proof: by duality
the
boundary components of $W$ are $2$-spheres, cap them off by
$3$-balls
to apply the Poincar\'e conjecture.
Decapsulating
back  the added $3$-balls
leads to the conclusion.] Therefore,
the $3$-ball and its holed avatars ({\it Swiss cheeses} or
{\it Grundformen} in the jargon of M\"obius
\cite{Moebius_1863}) are the only possible bags candidates,
for the construction of simply-connected $\omega$-bounded
$3$-manifolds.}
\end{rem}

Albeit far from proving it, Proposition~\ref{Lefschetz-fppf}
infers a certain evidence to
Conjecture~\ref{3D-hairy-ball-thm}, in the sense that any
corruption---if it exists---is instigated by a pipe lacking a
nice canonical exhaustion---a {\it wild} pipe, so to speak.

Even though Conjecture~\ref{3D-hairy-ball-thm} fails in
dimension $\ge 4$, Proposition~\ref{Lefschetz-fppf} gives
varied special cases, e.g., ${\Bbb S}^2\times {\Bbb L}^2$ (or
more generally $M\times {\Bbb L}^n$, where $M$ is a closed
manifold with $\chi(M)\neq 0$ and $n\ge 0$), for those
manifolds have a regulated growing mode (canonical exhaustion)
around a bag with non-zero Euler characteristic.

\begin{rem} {\rm Of course Lefschetz leads to somewhat
stronger results than those obtained via Brouwer, even when
each $W_\alpha$'s are contractible. For instance,
Proposition~\ref{Lefschetz-fppf} applies to the contractible
manifold-patch $W^4$ (of Freedman 1982 \cite{Freedman82})
bounding the Poincar\'e homology $3$-sphere $\Sigma^3$,
considered as the bag of the manifold obtained by capping off
the boundary by $\Sigma^3\times {\Bbb L}_{\ge 0}$.}
\end{rem}

\begin{rem} ($\omega$-bounded disruption of Heinz Hopf)\label{reverse-engineering} {\rm
One should not expect too much from a reverse engineering to
Proposition~\ref{Lefschetz-fppf}, where assuming each
$W_\alpha$ brushed
one hopes by a transfinite
assembly to gain a brush on $M$.
For instance,  on the connected-sum surface $M={\Bbb L}^2 \#
{\Bbb L}^2$ (two gemelar long planes ${\Bbb L}^2$ linked by a
worm hole), one cannot glue transfinitely the natural flows on
the annuli of the
canonical $\omega_1$-exhaustion. This would corrupt the global
topology of the pipe (compare
Proposition~\ref{Whitney's_flow-failure-on-plane_pipe}). With
some extra work, one can show that $M$ lacks any brush. This
involves a special argument similar to the one
exposed in Proposition~\ref{ad_hoc-long-quadrant} below.}
\end{rem}

\begin{rem} {\rm  Proposition~\ref{Lefschetz-fppf} applies as well in dimension
2, offering an alternate road to
Theorem~\ref{omega-bounded_Hairy-ball-thm}, at first glance at
least. However
this approach is again
plagued by the existence of
``wild''
pipes (cf. Nyikos~\cite[\S 6, p.\,669--670]{Nyikos84}). Hence
the Poincar\'e-Bendixson argument
(Theorem~\ref{omega-bounded_Hairy-ball-thm}) still incarnates
a wider range of applicability, for it applies
unconditionally (yet, just in the simply connected case!).
%
%

 }

\end{rem}

The following two hazardous conjectures
(the first would
follow from the second) are merely an avowal of our
ignorance:

\begin{conj}\label{hairy-ball-conj:crazy} An
$\omega$-bounded $n$-manifold $M$
with $\chi(M)\neq 0$ has the ffp for flows.
\end{conj}

\begin{conj}\label{Lefschetz-conj:crazy}
$\omega$-bounded $n$-manifolds have finite-dimensional
(singular) homology over the rationals, and are Lefschetz
spaces, i.e., any self-map with non-zero Lefschetz number
has a fixed point.
\end{conj}

\subsection{{\it Ad hoc} hairy balls
via
Cantorian
rigidity
}\label{ad_hoc:sec}

Theorem~\ref{omega-bounded_Hairy-ball-thm} does not apply to
${\Bbb L}_+^2={\Bbb L}_+ \times {\Bbb L}_+$ the square of the
open long ray ${\Bbb L}_+$, nor does
Proposition~\ref{Lefschetz-fppf} apply to ${\Bbb L}^2 \# {\Bbb
L}^2$ or to the long plane surmounted by a cylindrical tower
${\Bbb S}^1\times {\Bbb L}_{\ge 0}$, where
${\Bbb L}_{\ge 0}$ is the closed long ray.
Nonetheless, it turns out that all
these surfaces
lack a brush.
Basically, the reason is that on the planar pipe the
underlying foliation is---by the ``Cantor rigidity'' of
\cite{BGG}---forced to behave asymptotically like concentric
squares, yielding a global long cross-section whose
return times
falsify the global topology of the planar pipe to a
cylindrical one. (Recall
that a continuous real-valued function on the long ray is
eventually constant.) Thus, the ultimate reason is
again the form of Cantor
rigidity effecting that a cylindrical pipe differs radically
from a planar pipe. We detail this
{\it ad~hoc} rigidity argument in the first situation (the
other cases are similar):


\begin{prop}\label{ad_hoc-long-quadrant}
The long quadrant $Q={\Bbb L}_+^2$ does not support a
brush.
\end{prop}


\begin{proof} We divide
it in two steps:

{\sc Step 1: A special case.} {\it The
foliation $\cal F$ of $Q$ by short lines bifurcating when they
cross the diagonal of $Q$ (with leaves
of the form
$L_{\alpha}=(\{\alpha\}\times ]0,\alpha])\cup
(]0,\alpha]\times \{\alpha \})$,
$\alpha\in{\Bbb L}_+$)
does not come from a flow.}

By contradiction, assume the existence of a flow $f$ whose
phase-portrait induces $\cal F$. For each $\alpha\ge 1$, we
measure the time $\tau(\alpha)$ elapsed
until the point $(\alpha, 1)$
reaches the position $(1,\alpha)$. (This
is the time needed to travel from the horizontal cross-section
$\Sigma_h={\Bbb L}_{\ge 1}\times\{1\}$ to the vertical one
$\Sigma_v=\{1\}\times{\Bbb L}_{\ge 1}$; after a time-reversion
we assume the motion to have the ``right'' orientation.) The
function $\tau$ is real-valued, continuous and defined on
${\Bbb L}_{\ge 1}$, so
must be eventually constant.
This would imply that ${\Bbb L}_{\ge 1}\times {\Bbb L}_{\ge
1}$ is homeomorphic to a long strip $S={\Bbb L}_{\ge 0}\times
[0,1]$. This is however a
falsification of the global topology: indeed it is easy to
show that these surfaces are not homeomorphic, e.g. they can
be distinguished by looking at foliations on their doubles.
[Recall from~\cite{BGG} that the double $2S$, which is
$\Lambda_{0,1}:=({\Bbb S}^1\times {\Bbb L}_{\ge 0})\cup{\Bbb
B}^2$, has no foliation.]

{\sc Step 2: The general case.} We
shall reach a reduction to
Step 1. Let $f$ be a brush on $Q$, and look at the induced
foliation
(Theorem~\ref{folklore}). Cut $Q$ along the diagonal, and
apply
\cite{BGG} to conclude that  each octant
$O_i$ $(i=1,2)$ is (asymptotically) foliated by straight long
rays or by straight short segments. The first option is to be
excluded for a flow,
%
%
%
%
hence both octants $O_{i}$ are
foliated by short segments occurring for a club $C_i\subset
{\Bbb L}_+$.
Passing to their common intersection $C=C_1\cap C_2$, we find
a new club where all the segments piece nicely together along
the diagonal. A
resemblance with case of Step 1 is emerging, except for a lack
of complete ``straightness'' of the foliation. Cutting the
foliation
along the
broken lines $L_{\alpha}$ for $\alpha \in C$, we get
subregions which are strips $[0,1]\times {\Bbb R}$. Typically
we may meet a Reeb foliation,
trouble-shooting the
well-definition of the map $\tau$:  orbits starting from the
horizontal cross-section $\Sigma_h$ instead of reaching the
vertical one $\Sigma_v$, will loop-back like boomerangs.
The issue
is that Reeb components cannot be too numerous, for an
infinitude of them implies
a clustering onto a singularity front, imposed by a visual
compression under the Cantor perspective
(use the fact that an increasing $\omega$-sequence in ${\Bbb
L}_+$ is convergent).
Jumping over
these finitely many Reeb components by looking sufficiently
far away (say $\alpha \ge \alpha_0$) we may define the
function $\tau\colon {\Bbb L}_{\ge \alpha_0} \to {\Bbb R}$
to conclude as in Step~1.
\end{proof}
\subsection{Abstract separation yoga
via the five lemma}

We shall now redirect our attention to the pseudo-Moore
problem (Question~\ref{Pseudo-Moore:conj:question}) of showing
that a surface sharing
{\it in abstracto} the
distinctive features of the Moore surface lacks a brush. To
prepare the terrain, this section trains some abstract
non-sense ``Jordan'' separation yoga.
%
Below, singular homology  is understood and coefficients are
in ${\Bbb Z}$.

\begin{lemma}\label{five_lemma} Let $J$ be closed set of a space $M$. Assume
$J \subset U$ strictly contained in an open set $U$ of $M$.

{\rm (i)} If $J$ divides $M$, then $J$ divides $U$.

{\rm (ii)} If $J$ divides $U$ and if the map $H_1(U)\to
H_1(M)$ is surjective, then $J$ divides $M$.
\end{lemma}

\def\ziehen{\hskip-6pt}
\begin{proof}
Superpose the two exact sequences of the pairs $(U,U-J)$ and
$(M,M-J)$:
$$
\begin{matrix}
&H_1(U)\ziehen &\rightarrow\ziehen &H_1(U,U-J)\ziehen
&\rightarrow\ziehen& H_0(U-J)\ziehen &\rightarrow\ziehen&
H_0(U)\ziehen &\rightarrow\ziehen& H_0(U,U-J)=0  \cr

&\downarrow & &\downarrow & &\downarrow &  &\downarrow  &
&\downarrow \cr

&H_1(M)\ziehen&\rightarrow\ziehen &H_1(M,M-J)\ziehen
&\rightarrow\ziehen& H_0(M-J) \ziehen &\rightarrow\ziehen&
H_0(M)\ziehen &
 \rightarrow\ziehen& H_0(M,M-J)=0
\end{matrix}
$$
The second down-arrow is
isomorphic by excision.
Recall the {\it five lemma}:
suppose that the diagram of abelian groups has exact rows and
each square is commutative:
$$
\begin{matrix}
&C_1\ziehen &\rightarrow\ziehen&C_2\ziehen
&\rightarrow\ziehen& C_3\ziehen &\rightarrow\ziehen&
C_4\ziehen &\rightarrow\ziehen& C_5 \cr

&\quad\downarrow  f_1 &&\quad\downarrow  f_2 &
&\quad\downarrow f_3 & &\quad\downarrow f_4 & &\quad\downarrow
f_5 \cr

&D_1\ziehen &\rightarrow\ziehen &D_2\ziehen
&\rightarrow\ziehen & D_3\ziehen &\rightarrow\ziehen&
D_4\ziehen &
 \rightarrow\ziehen& D_5
\end{matrix}
$$
\vskip-5pt \noindent Then

(1) if $f_2$ and $f_4$ are epimorphisms and $f_5$ is a
monomorphism, then $f_3$ is an epimorphism.

(2) if $f_2$ and $f_4$ are monomorphisms and $f_1$ is an
epimorphism, then $f_3$ is a monomorphism.

\noindent Conclude by observing that (1) proves (i), while (2)
establishes (ii).
\end{proof}

\subsection{Back to the pseudo-Moore
question:
separable hairy ball theorem}

The following is another stagnation result, dual to
Theorem~\ref{Brouwer}, as it applies to separable manifolds
(separability is the
{\it d\'esincarnation} of $\omega$-boundedness as the
conjunction
of both point-set properties collapses to compactness).

\begin{theorem}\label{hairy_ball_thm:separable} Any flow
on a simply-connected, separable, non-metric,
non-bordered
surface has a
fixed point.
\end{theorem}

\begin{proof} By contradiction, let $f\colon
{\Bbb R} \times M\to M$ be a non-singular flow on such a
surface $M$. By separability, let $D$ be a countable dense
subset of $M$. By phagocytosis, like any countable subset of a
manifold, $D$ can be engulfed inside a chart $U$, see
\cite[Prop.\,1]{Gauld_2009}. Let ${\Bbb R} U=f({\Bbb
R}\times U)$ be the orbit of that (dense) chart $U$. The set
${\Bbb R}U$ is open, invariant (under the flow) and connected
(as a continuous image of the set ${\Bbb R}\times U \approx
{\Bbb R}^3$). Moreover ${\Bbb R} U$ is Lindel\"of, hence
cannot exhaust all of $M$
(which is non-metric).
%
%
%
Choose a point $x\in M-{\Bbb R} U$. A Poincar\'e-Bendixson
argument shows the orbit ${\Bbb R}x$ to be closed.
Besides, ${\Bbb R}x$ cannot be a circle,
for a fixed point
would be created by the non-metric Schoenflies
\cite{GaGa2010}, plus Brouwer. Hence ${\Bbb R}x$ is a line
(properly) embedded as closed set. 
Let $T$ be a tubular neighbourhood of ${\Bbb R}x$ (as usual
constructed within a Lindel\"of neighbourhood of the
Lindel\"of set ${\Bbb R}x$ by
metric combinatorial methods). This tube $T$
has a bundle structure over a base which is contractible, so
is trivial (Ehresmann-Feldbau-Steenrod)\footnote{Here it is
crucial that the base is metric (paracompact) as shown by the
tangent bundle to the contractible classical version of the
Pr\"ufer surface ($P$ collared). (Compare
Spivak~\cite{Spivak}.)}. In particular ${\Bbb R}x$ divides
$T$, so by Lemma~\ref{five_lemma}~(ii)
${\Bbb R}x$
divides $M$. As ${\Bbb R} U \subset M- {\Bbb R}x$, the
connectedness of ${\Bbb R} U$ implies its containment in one
component of $M- {\Bbb R}x$, but then ${\Bbb R} U$ fails to be
dense in $M$. This contradiction completes the proof.
\end{proof}

Albeit
quite general, the fixed-point theorems obtained so far
(Theorems \ref{omega-bounded_Hairy-ball-thm} and
\ref{hairy_ball_thm:separable}) have the serious limit\-ation
that they only apply to the simply-connected case. In the
non-simply-connected case we may sometimes appeal to
Proposition~\ref{Lefschetz-fppf}, at least if the long pipes
have a nice canonical exhaustion, and dually appeal to
Theorem~\ref{CTD:thm} if there is a cytoplasmic decomposition.
Let us however
not miss the following:

\begin{cor}\label{separable_hairy_ball_pi_1_ctble:cor} A
separable non-metric surface (boundaryless) with countable
fundamental group
lacks a brush.
\end{cor}

\begin{proof} Given a flow on such a surface,
lift it to
the universal covering
(still separable) and
apply (\ref{hairy_ball_thm:separable}).
\end{proof}

\section{Transitive flows
}

Beside non-singular flows (brushes), another weakening of
minimality
is {\it transitivity}:

\begin{defn} {\rm
A flow on a topological space is \emph{transitive} if it has
at least one dense orbit. A space capable of a transitive flow
is said to be \emph{transitive}, and \emph{intransitive}
otherwise.}
\end{defn}

\subsection{Transitivity obstructions
(non-separability and dichotomy)}

Obviously, a transitive flow implies separability of the
phase-space (look at ${\Bbb Q}x$ the rational times of a point
with dense orbit).
Thus, ${\Bbb L}^2$ or  the collared Pr\"ufer surface, $P_{\rm
collar}$, are certainly intransitive.
This does not inform us
on the status of $2P$, the doubled Pr\"ufer surface.
Yet, for surfaces we have
the classical  Poincar\'e-Bendixson obstruction: 

\begin{lemma}\label{dicho_implies_intransitive} A dichotomic
surface (i.e., divided by any embedded circle)
is intransitive.
\end{lemma}

\begin{proof} (Again the
Bendixson
sack argument.) By contradiction, let $x\in M$ have a dense
orbit under the flow $f$.
Draw a cross-section $\Sigma_x$ through $x$ and consider an
associated flow-box $f([-\varepsilon,\varepsilon] \times
\Sigma_x)$. The point $x$ must eventually return to
$\Sigma_x$, so the piece of trajectory from $x$ to its first
return, $x_1$, closed up by the arc $A$ of $\Sigma_x$ joining
$x$ to  $x_1$, defines a circle $J$ in $M$.
The component of $M-J$ containing the near future of $x_1$
(e.g. $f(\varepsilon / 2, x_1)$) will contain the full future
of $x_1$. Then the ``short'' past of the arc $A$ namely
$f(]-\varepsilon,0[\times {\rm int} A)$ is an open
``subrectangle'' of the flow-box
which the orbit of
$x$ will never
revisit again. A contradiction. 
\end{proof}

This shows the intransitivity of $2P$ which is clearly
dichotomic (see
the next section for a
detailed argument). In particular, simply-connected surfaces
are dichotomic (by the non-metric Jordan curve theorem in
\cite{GaGa2010}\footnote{Recall that the Hausdorff separation
axiom is crucial, for
 it is easy to draw on the {\it branched plane} a circle which does
not disconnect it. (The latter arises from $2$ replicas of
${\Bbb R}^2$ by gluing along an open half-plane.)}),
hence intransitive. This holds for the Moore surface and the
{\it Maungakiekie} (i.e., the result of a long cytoplasmic
expansion of the $2$-cell, as discussed in
(\ref{cytoplasmic_expansions})), etc.

\subsection{Dichotomy:
heredity and Lindel\"of approximation}

This section collects some basic
lemmas on {\it dichotomy} (i.e., the Jordan curve theorem
holds true globally), establishing in particular the dichotomy
of the doubled Pr\"ufer surface, $2P$, i.e.
Calabi-Rosenlicht's version \cite{Calabi-Rosenlicht_1953}. The
reader with a good visual acuity can safely skip this section
without loss of continuity.

\begin{lemma}\label{DICH-heredity} Any open set
of a dichotomic surface is itself
dichotomic.
\end{lemma}

\begin{proof} This follows
at once  from Lemma~\ref{five_lemma} (i).
\end{proof}

The converse---in order to be non-tautological---takes the
following form:

\begin{lemma}\label{DICH-Lindeloef-approx}
A surface each of whose Lindel\"of subsurfaces are dichotomic
is itself dichotomic.
\end{lemma}

\begin{proof}\footnote{Unfortunately
one cannot take advantage of
Lemma~\ref{five_lemma} (ii).} This argument
has some variants
depending on the amount of geometric topology
inferred [we
bracket the more geometric variants]. Assume by contradiction
the existence of a Jordan curve $J\subset M$ in the surface,
$M$,  such that $M-J$ is connected. Choose $U$ a Lindel\"of
subsurface with $U\supset J$. [We could take for $U$ a tube
around $J$.] The set $U-J$ has at most countably many
components $U_i$, $i=1,2,\dots$. [If $U$ is a tube then
exactly two components.] Pick a point $x_i\in U_i$ in each
component. Consider the countable set $C=\{x_1,x_2, \dots\}$
in the (connected) manifold $M-J$. By arc-wise connectivity,
there is pathes $c_i$ from $x_1$ to $x_i$. Covering by charts
the union of those pathes, $C$ can be engulfed in  a connected
Lindel\"of open set $L\subset M-J$. [One could take for $L$ a
chart (taking advantage of the phagocytosis lemma).]
The set $U\cup L$
is open and Lindel\"of. We have $(U\cup L)-J=(U-J) \cup L
=\bigcup_{i\in {\Bbb N}} U_i \cup L$. Regarding this union as
$L$ plus the sets $U_i$ (each meeting $L$), it follows (by
general topology) that this ``bouquet-like'' union is
connected. This contradicts
our assumption of ``Lindel\"of dichotomy''. [In the variant
where $U$ is a tube,
we also find $L\subset M-J$ a Lindel\"of connected surface
containing $\{x_1,x_2\}$: cover by charts (of $M-J$) the path
$c_2\subset M-J$ (joining $x_1$ to $x_2$) while keeping only
the component of this union of charts which contains $c_2$.
Again $(U\cup L)-J=(U-J) \cup L =U_1\cup U_2 \cup L$, which is
connected,
as the union of two connected sets
with a common intersection.]
\end{proof}

\begin{cor}\label{Pruefer_dichotomic}
The doubled Pr\"ufer surface $2P$ is dichotomic.
\end{cor}

\begin{proof} By Lemma~\ref{DICH-Lindeloef-approx} it is
enough to show that each Lindel\"of subsurface $L$ of $2P$ is
dichotomic. Consider the open covering of $2P$ by the sets
$B_x$ consisting of both half-planes plus the pencil of rays
through $x\in {\Bbb R}$. By  Lindel\"ofness of $L$, one may
extract a countable subcover $B:=\bigcup_{x\in C} B_x\supset
L$.
It is
plain that $B$ embeds in the plane ${\Bbb R}^2$,
and the dichotomy of $L$ follows from
Lemma~\ref{DICH-heredity} (plus the classical Jordan curve
theorem).
\end{proof}

\begin{rem} {\rm Such routine separation arguments
are of course very easy, 
and overlap
a remark  by R.\,L. Moore, as reported
by F. Burton Jones \cite[p.\,573, Sec.\,4,
Parag.\,2]{Jones_F._Burton_1966},
where it is observed that the Moore surface $M$ is
``globally'' Jordan, i.e. dichotomic. (This can be checked
along the same
lines as what we  did for $2P$.)}
\end{rem}

\subsection{Surfaces with prescribed topology and
dynamics}\label{Geography}

As we saw ``simple'' topology
oft  impedes
``complicated'' dynamics (e.g., dichotomy obstructs
transitivity). This section works out the experimental side of
the various topologico-dynamical interactions.
As usual, examples are intended to test the
exhaustiveness of the theoretical obstructions listed so far.
Arguably, any knowledge of the world (resp. theory) starts and
ends with experiments (resp. examples).
%
%
%
An oblique hope is that
a blend of topologico-dynamical prescriptions singles out
subclasses in the jungle of (non-metric) surfaces, where a
classification looks more tractable. This scenario will rarely
happen, yet when, in the nihilist art-form of an empty-set
{\it classifiant}.
%
%
%
%
%
So we start by a selection of:

\noindent $\bullet$ {\it Topological attributes} including:
metric, separable, simply-connected, dichotomic; {\it versus},

\noindent $\bullet$ {\it Dynamical attributes} including:
minimal,
quasi-minimal\footnote{A flow is {\it quasi-minimal} if it has
a finite number of fixed-points, while all non-stationary
orbits are dense.}, non-singular, transitive.

Below, we
have pictured
a {\it Venn diagram} showing the mutual disposition of these
subclasses inside the universe of all Hausdorff surfaces. We
shall primarily ask  for representatives in each subclasses,
and secondarily for a classification if possible. Some
accompanying comments on this diagram are in order:

\begin{figure}[h]
\centering
    \epsfig{figure=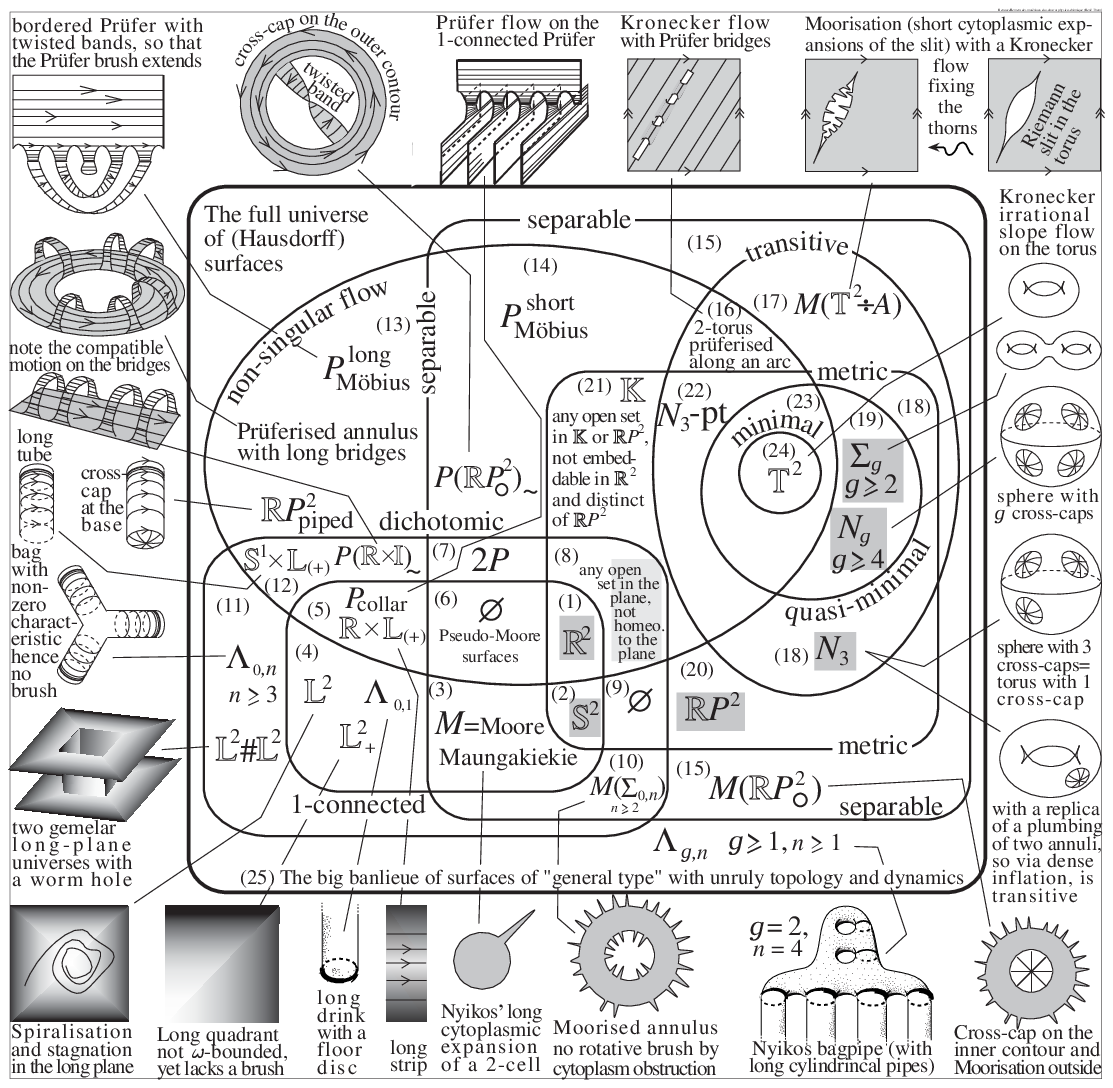,width=172mm
    }
\end{figure}

(a) ``Rounded rectangles'' correspond to topological while
``ovals'' to dynamical attributes.


(b) For
identifying specimens
the following symbolism is employed: $P( \cdot )$ denotes the
Pr\"uferisation operator, $M(\cdot)$ the Moorisation. (This is
merely a matter of globalising the classic constructions of
Pr\"ufer and Moore, cf. eventually
Definition~\ref{Pruefisation-Moorisation}.) Thus, the bordered
Pr\"ufer $P$ is $P({\Bbb H})$ the Pr\"uferisation of the upper
half-plane ${\Bbb H}={\Bbb R} \times {\Bbb R}_{\ge 0}$, while
${M}$ the classic Moore surface is $M({\Bbb H})$. A ``circle
index'' means excising a $2$-disc. A ``tilde index'' means
that some identifications are made (usually on the boundary of
a Pr\"uferisation).

(c)
Notation: ${\Bbb R}$
the real line, ${\Bbb S}^1$ the circle, ${\Bbb I}=[0,1]$ the
interval, ${\Bbb L}_+=]0,\omega_1[$ and ${\Bbb L}_{\ge
0}=[0,\omega_1[$ are the open (resp. closed) long rays, ${\Bbb
L}$ the long line, ${\Bbb S}^2$ the 2-sphere, ${\Bbb
T}^2={\Bbb S}^1\times {\Bbb S}^1$ the 2-torus,
$\Sigma_g$
the closed orientable surface of genus $g$, $\Sigma_{g,n}$ the
same with $n$ holes ($2$-discs excisions), $N_g$ the closed
non-orientable surface of genus~$g$ (defined in accordance
with Riemann as the maximal number of disjoint non-dividing
circles),
thus $N_g$ is the sphere with $g$ cross-caps (hence
$\chi=2-g$). In particular, ${\Bbb R}P^2=N_1$ is the {\it
projective plane}, ${\Bbb K}=N_2$ is
the {\it Klein bottle} (non-orientable closed surface with
$\chi=0$).
Finally $\Lambda_{g,n}$ denotes the genus $g$ surface with $n$
pipes modelled on the cylinder ${\Bbb S}^1 \times {\Bbb
L}_{\ge 0}$.

(d) A shaded ``manifold symbol'' refers to the issue that the
given manifold(s) turns out to be the {\it unique}
representative(s) in the given class. The empty set symbol
$\varnothing$ indicates a class lacking any representative
(we exclude the empty set to be a genuine surface).

\smallskip
The uniqueness of ${\Bbb S}^2$ and ${\Bbb R}^2$ in their
respective classes follows from the classification of
simply-connected  metric surfaces.
%
Recall the following key result (incarnating an advanced form
of Poincar\'e-Bendixson theory beyond dichotomy):

\begin{lemma}\label{Markley:prop} Among closed
surfaces only ${\Bbb S}^2, {\Bbb R}P^2$ and ${\Bbb K}$ (Klein
bottle) are intransitive.
\end{lemma}

\begin{proof} The intransitivity of
${\Bbb S}^2$ and ${\Bbb R}P^2$ follows from
Poincar\'e-Bendixson (after lifting the flow to the universal
covering in the second case). The intransitivity of
${\Bbb K}$ was first established by Markley 1969
\cite{Markley_1969} (independently Aranson 1969),
yet the argument of Guti\'errez 1978~\cite[Thm~2,
p.\,314--315]{Gutierrez_1978_TAMS} seems to be
inter\-continentally recognised as the ultimate
simplification. The transitivity of all remaining closed
surfaces is observed in Peixoto
1962~\cite[p.\,113]{Peixoto_1962},
also Blohin 1972~\cite{Blohin_1972}. Beside the meticulous
surgeries used by those
authors, it is
pleasant to
recall the
following cruder approach.
Since a cross-cap diminishes $\chi$ by one unit, we have the
relation $N_3={\Bbb R}P^2\#{\Bbb R}P^2 \#{\Bbb R}P^2\approx
{\Bbb T}^2 \# {\Bbb R}P^2$. Thus all other closed surfaces
contain a replica of the plumbing of two open annuli (i.e. a
punctured torus). Granting some geometric intuition,
this
open set can be inflated until to be dense.
(Section~\ref{Inflation:sec} below discusses the
issue that such dense inflations of open non-void sets
might always be  possible in separable manifolds.) By
extending a Kronecker flow on this inflated punctured torus to
the
ambient closed surface one obtains the desired transitive
flow. (This mechanism is further discussed below in
Section~\ref{const_trans_flows:gnal_recipe:sec}.)
\end{proof}

The uniqueness of ${\Bbb R}P^2$ in its class is now easy. We
seek after surfaces without brush but metric. Since open
metric surfaces support brushes\footnote{In fact this result
holds in any dimension, and even in the topological
category~(Lemma~\ref{Morse-Thom}). In the $2$-dimensional case
the argument simplifies: introduce a smooth structure, find a
Morse function without critical points (if any kill them by
excising an arc starting from the singularity and running to
infinity), and conclude by taking its gradient flow.}, our
surface must be compact. Intransitivity leaves
{\it only} the three possibilities listed in
Lemma~\ref{Markley:prop}. Non-dichotomy excludes ${\Bbb S}^2$,
while ${\Bbb K}$ is ruled out by the ``no brush'' condition,
leaving
${\Bbb R}P^2$ as the unique
solution.

The emptiness of the class lying between ${\Bbb S}^2$ and
${\Bbb R}P^2$ (label (9) in our Venn diagram) is argued
similarly. Such a surface must be compact (else it has a
brush).
Yet, the only closed dichotomic
surface is ${\Bbb S}^2$.

We now embark in a more systematic exploration of our diagram
by starting from the ${\Bbb R}^2$ region, while spiraling
clockwise (mimicking the numbering of parisian
arrondissements):

\smallskip
(1) The first arrondissement is chosen as
the one of ${\Bbb R}^2$, for the plane is not only locally but
globally Euclidean. {\it The plane is characterised as the
unique metric $1$-connected surface bearing a brush.} Its only
drawback is a certain dynamical poorness (transitivity is
impeded by its dichotomy).

\smallskip
(2) Moving below we find ${\Bbb S}^2$,
which
is {\it the unique metric $1$-connected surface lacking a
brush.} Since the neighbouring class (9) is empty the
statement can be sharpened into: {\it The $2$-sphere is the
unique dichotomic metric surface lacking a brush.}

\smallskip
(3) On the left of ${\Bbb S}^2$, we encounter the Moore
surface ${ M}$ and the {\it Maungakiekie} (i.e., one long
cytoplasmic expansion of the $2$-cell, recall
(\ref{cytoplasmic_expansions})). (Both lack a brush as they
have a CTD, cf. Theorem~\ref{CTD:thm}.)

\smallskip
(4) Still more to the left we have ${\Bbb L}^2$ and
$\Lambda_{0,1}$.  The long plane ${\Bbb L}^2$ has no foliation
of dimension $1$ by short leaves, while $\Lambda_{0,1}$ has no
foliation at all \cite{BGG}, so both do not accept a brush.
This follows also from
Theorem~\ref{omega-bounded_Hairy-ball-thm} or
Corollary~\ref{fppf-Ln}. In view of Nyikos
\cite[p.\,669]{Nyikos84}
this class contains a bewildering variety of
specimens of cardinality $2^{\aleph_1}$ of which the two above
are just the most civilised examples. This class also includes
non-$\omega$-bounded examples, e.g. ${\Bbb L}_{+}^2$
(Proposition~\ref{ad_hoc-long-quadrant}).

\smallskip
(5) Moving up, we meet ${\Bbb R}\times {\Bbb L}_{(+)}$ (the
parenthetical ``plus'' means that we may take either the long
line ${\Bbb L}$ or the long ray ${\Bbb L}_+$) and also $P_{\rm
collar}:=P\cup (\partial P \times {\Bbb R}_{\ge 0})$,
i.e. the original Pr\"ufer surface which has a brush
(\ref{Pruefer_flow}). (In view of
Theorem~\ref{omega-bounded_Hairy-ball-thm}, there is no
$\omega$-bounded examples in this class.)

\smallskip
(6) This is an empty region corresponding to the pseudo-Moore
problem (solved via Theorem~\ref{hairy_ball_thm:separable}.)

\smallskip
(7) Here we have $2P$ the doubled Pr\"ufer surface which has a
brush
(\ref{Pruefer_flow}) and is dichotomic
(\ref{Pruefer_dichotomic}).
What else? (Try puncturing.)

\smallskip
(8) This class contains for instance the punctured plane
${\Bbb R}^2-\{0\}$ and more generally any open set of the
plane (topologically distinct from the plane). This is a
complete list of representatives due to the classification of
(dichotomic) metric surfaces (compare e.g.,
Ker\'ekj\'art\'o \cite{Kerekjarto_1923}).

\smallskip
(9) This class is empty, as already argued.

\smallskip
(10) Here we have $M(\Sigma_{0,n})$ for $n\ge 2$ (recall that
$\Sigma_{g,n}$ denotes the compact orientable surface of genus
$g$ with $n$ boundary components and that $M$ is the
Moorisation operation). For $n=2$ this surface is an annulus
$\Sigma_{0,2}=S^1 \times [0,1]$ Moorised along its boundary.
Since the genus $g$ is zero these surfaces are dichotomic
(apply Lemma~\ref{DICH-Lindeloef-approx}), and they lack a
brush (as they have a CTD). [Note that $M(\Sigma_{0,1})$
belongs to arrondissement (3) being perhaps homeomorphic to
the
classic Moore surface ${ M}=M({\Bbb H})$, where ${\Bbb
H}={\Bbb R}\times{\Bbb R}_{\ge 0}$.]

\smallskip
(11) Now we have the surfaces $\Lambda_{0,n}$ for $n\ge 3$.
These surfaces having genus $0$ are dichotomic (again
Lemma~\ref{DICH-Lindeloef-approx}), with non-trivial $\pi_1$
(as soon as $n\ge 2$) and finally lack a brush (either because
in \cite{BGG}
it was shown that $\Lambda_{g,n}$ has a foliation only when
$(g,n)$ is $(1,0)$ or $(0,2)$ or alternatively by
Proposition~\ref{Lefschetz-fppf}). This class also contains
the surface ${\Bbb L}^2 \# {\Bbb L}^2$, as discussed in
Section~\ref{ad_hoc:sec}.

\smallskip
(12) In this class we have ${\Bbb S}^1\times {\Bbb L}_{(+)}$.
What else? Again puncturing works.
More sophisticated examples are obtainable by ``rolling
around'' the tangent bundle of a smooth structure on ${\Bbb
L}_{(+)}$. If $L$ is a non-metric smooth 1-manifold, remove
from its tangent bundle $TL$ the zero-section to get two
components
$TL^+$, $TL^{-}$. Scalar multiplication by a positive
real $\lambda\neq 1$ on $TL^+$ yields a ${\Bbb Z}$-action on
$TL^+$, whose quotient $S:=TL^+ / {\Bbb Z}$ is a circle-bundle
over $L$. We have then a brush $f\colon {\Bbb R} \times S \to
S$ induced by $\phi:{\Bbb R}\times TL^+ \to TL^+$ given by
$\phi(t, v)=\lambda^t v$,
which induces an action of the circle ${\Bbb R}/ {\Bbb Z}$ on
$S$.
By
Riemannian geometry
the tangent
bundle to a smooth non-metric manifold cannot be trivial
\cite{Morrow_1969},
adumbrating that the surfaces $S$ are not homeomorphic to the
trivial circle bundle over $L$. Another example is the following: start
with a strip ${\Bbb R}\times [0,1]$ and Pr\"uferise its
boundary, and then glue long bands to link boundary components
with the same first coordinate. This surface, denoted $P({\Bbb
R}\times{\Bbb I})_{\sim}$, is not separable (because of the
``longness'' of the bands), is dichotomic
(Lemma~\ref{DICH-Lindeloef-approx}) (and incidentally with a
$\pi_1$ bigger than the previous ones).

\smallskip
(13) Here we give $3$ examples: (1) start from $P$ the
bordered Pr\"ufer surface $P=({\Bbb R}\times {\Bbb R}_{>0})
\sqcup \bigsqcup_{x\in {\Bbb R} } R_x$ endowed with its
natural brush (windscreen wiper flow of
Proposition~\ref{Pruefer_flow}). For each non-zero real $x\in
{\Bbb R}-\{ 0\}$, we glue the boundary components $R_x$ of $P$
with the opposite $R_{-x}$ in a way consistent with the flow.
This produces many embedded copies of  the M\"obius band, so
denote this surface by $P_{\rm M\ddot{o}bius}^{\rm short}$. To
make the result non-separable we attach
``long'' M\"obius bands (e.g. by first adding to $P$ a closed
collar $\partial P \times [0,1]$ and then performing the
``flow compatible'' identifications).
Finally, aggregate an open collar to the ``central'' boundary
component $R_0$. The resulting surface has a brush, is not
separable and not dichotomic (it contains a M\"obius band),
hence non-orientable. (2) To get an orientable example
Pr\"uferise the annulus and link (radially) related components
by ``bridges'' homeomorphic to ${\Bbb R} \times [0,1]$. (3)
Yet another example is the surface deduced from ${\Bbb
S}^1\times {\Bbb L}_{\ge 0}$ by identifying via the antipodal
map (cross-capping) the boundary. (This is ${\Bbb R}P^2$ with
a long pipe of the cylinder type.) This surface, call it
${\Bbb R}P^2_{\rm piped}$, has a brush (in fact a circle
action) and is not dichotomic (the image of the boundary in
the quotient is a non-dividing circle). [From its description
as ${\Bbb R}P^2$ with a long (cylindrical) pipe, this surface
can be shown to be
{\it universally intransitive}, i.e. none of its open subset
is transitive.]

\smallskip
(14) Here it is a bit more difficult to find examples. We will
turn back to this after studying case (15).
A natural candidate could be the surface $P_{\rm
M\ddot{o}bius}^{\rm short}$ constructed in the first part of
(13); yet being separable, it is difficult to ensure
intransitivity.
(However using the argument of (14bis) below, intransitivity
will be clear.) [Another candidate could be a Pr\"uferised
annulus with short bridges, yet as this contain a
replica of two plumbed annuli it is more likely that this
example is transitive so belongs to (16).]

\smallskip
(15)
Remove from ${\Bbb R}P^2$ the interior of a closed 2-disc (to
get a M\"obius band)
and Moorise the boundary to obtain $M({\Bbb R}P^2_{\circ})$
(the ``circle index'' is a disc excision).
This surface (call it $S$) has a CTD, hence lacks a brush and
is clearly separable. For
belonging to class (15) it remains to prove intransitivity. By
contradiction, let $f\colon{\Bbb R} \times S \to S$ be a
transitive flow. The surface $S$ has a decomposition $S=U
\sqcup \bigsqcup_{x\in S^1} T_x$, where $U$ is an open
M\"obius band (=punctured ${\Bbb R}P^2$), and each $T_x$ is
homeomorphic to ${\Bbb R}_{\ge 0}$. Let $x\in S$ be a point
with dense orbit ${\Bbb R}x=f({\Bbb R} \times \{x \})$, and
choose $V$ a chart around $x$. Its orbit ${\Bbb R}V:=f({\Bbb
R} \times V)=\bigcup_{t\in {\Bbb R}} f_{t} (V)$ is open,
Lindel\"of and transitive under the restricted flow $f\colon
{\Bbb R} \times {\Bbb R}V \to {\Bbb R}V$. The open cover of
$S$ by the sets $U_x:=U \cup T_x$ shows that ${\Bbb R}V$ is
contained in a countable union $\bigcup_{x\in C} U_x$, where
$C \subset S^1$ is countable). It is easy to show that
$\bigcup_{x\in C} U_x$ which is the core plus countably many
thorns
remains homeomorphic to the original core
$U$, which is an (open) M\"obius band
(apply Morton Brown's theorem). Now, the transitivity of
${\Bbb R}V$
violates the universal intransitivity of ${\Bbb R}P^2$, i.e.
all its open subsets are intransitive (cf.
Lemma~\ref{intransitivity} below).

\smallskip
(14bis) [(14) revisited!] Separability makes hard to ensure
intransitivity, yet we use the same trick as in (15) by taking
advantage of the universal intransitivity of ${\Bbb R}P^2$.
Thus, start with the projective plane visualised as a closed
2-disc modulo antipodes on the boundary. Remove the interior
of a central disc to get $W$ a surface with one boundary
circle (a M\"obius band). Consider the flow given by a
rotational motion on this $W$ (annulus with external circle
identified by antipodes). Pr\"uferise the intern circle, and
consider a Pr\"ufer flow. Then glue diametrically opposite
boundaries in the way prescribed by the flow to obtain our
surface $S$. (By construction it has a brush and is
separable.) To check intransitivity, we argue by contradiction
as before. Choose a chart $V$ around a point $x$ with a dense
orbit, and note that  $f({\Bbb R} \times V)$ is {\it a
fortiori} transitive and Lindel\"of, hence contained in the
``core'' ${\rm int}(W)$ plus countably many ``bridges''.
Denote by $S_\omega$ this ``countable approximation'' of $S$.
Attaching a single
twisted band amounts to a single puncturing (if there were no
twist this would produce {\it two} punctures!). [This can also
be checked either by cut-and-past or
via the classification
of compact surfaces, after
aggregating the natural boundary.] Arguing inductively
$S_{\omega}$ is in fact homeomorphic to a $\omega$ times
punctured M\"obius band;  against the universal intransitivity
of ${\Bbb R}P^2$.

\smallskip
(16) A simple example is the 2-torus Pr\"uferised along an
arc. This surface has a transitive brush deduced  from a
windscreen wiper motion (as suggested in
Example~\ref{Kronecker-Pruefer}). Yet, we promised a formal
treatment of (\ref{Kronecker-Pruefer}) and this relies on
Lemma~\ref{Pruefer_flow:generalised:lemma} below. We apply the
latter to $W$ the result of a {\it Riemann slit} along a piece
of orbit of a Kronecker flow ({\it slitting} merely
amounts to duplicate each interior point  of the segment). We
take care of removing the two extremities of the ``slit''.
Thus, $W$ is non-compact and with {\it two}
boundary-components (``lips''). Equip $W$ with the Kronecker
flow suitably slowed down by multiplying its velocity vector
field by a smooth positive function vanishing precisely on the
two lips. Further, arrange a linear decay (of speed) when
approaching the ``lips'', then case (1) of
Lemma~\ref{Pruefer_flow:generalised:lemma} gives the required
flow (after piecing together the boundaries of the
Pr\"uferisation $P(W)$ lying ``opposite'', i.e., those which
were indexed by the same point prior to the slit).

\smallskip
(17) 
Again we just
show an example:
start with the torus ${\Bbb T}^2={\Bbb R}^2/ {\Bbb Z}^2$ with
an irrational flow $f$.
Take a portion of trajectory $A=f([t_1,t_2]\times\{x\})$ (say
contained in the fundamental domain). Slit (\`a la Riemann)
the torus along this arc $A$ to get a bordered surface
$W={\Bbb T^2} \div A$ (imagine again that points of the
interior of the arc are duplicated).
Then  Pr\"uferise $W$ to get $P(W)$ and finally Moorise to get
the surface $M(W)$ which
lacks a brush because it has a CTD. It remains to find a
transitive flow. This involves the idea that if one alter the
Pr\"ufer flow (which has a linear speed decay when approaching
the
boundary), into one having a quadratic speed decay one obtains
a quadratic Pr\"ufer flow fixing point-wise the boundary of
$P$ and so descends on the Moore surface by fixing the
``thorns'' (in the classic case an explicit formula is
$f(t,(x,y))=(x+ty^2,y)$). In view of the differential
geometric character of the Pr\"ufer construction (think with
``rays''), this construction clearly globalises. The following
two items should throw more light on this aspect:

\begin{defn} (Pr\"uferisation--Moorisation)\label{Pruefisation-Moorisation}
{\rm Given a bordered metric surface $W$ (with a smooth
structure and a Riemannian metric). One defines its {\it
Pr\"uferisation} $P(W)={\rm int} W \sqcup \bigsqcup_{x\in
\partial W} R_x$  by aggregating to the interior of
$W$, all the ``interior'' rays in the tangent $2$-planes, $T_x
W$, with $x\in \partial W$.
A topology on $P(W)$ is introduced by mimicking the Pr\"ufer
topology,
making $P(W)$ into a bordered non-metric surface whose
boundary components are the sets $R_x$ each homeomorphic to
${\Bbb R}$.
The {\it Moorisation} $M(W)$ {of  $W$ } refers to the
(boundaryless) surface, quotient of $P(W)$ by self-gluing each
of its boundary-components $R_x$ via
the involution given by reflecting ``rays'' about the ray at
$x$ orthogonal to the boundary.}
\end{defn}

\begin{lemma}\label{Pruefer_flow:generalised:lemma} (Generalised Pr\"ufer flows) Given a smooth
flow $f$ on $W$ fixing point-wise $\partial W$, there is a
canonically
induced ``Pr\"ufer flow'', denoted $P(f)$, on the
Pr\"uferisation $P(W)$.
Two special
cases are of
interest:---{\rm (1)} if the flow $f$ has a linear speed decay
when approaching the boundary then geodesic-rays
undergo a ``windscreen wiper motion'' via $f$, while---{\rm
(2)} if this decay is quadratic then geodesic-rays
deform into parabolas
keeping the same tangent. Thus, in case {\rm (1)} the $P(f)$
has no fixed point (on $\partial P(W)$), while in case {\rm
(2)} all points of $\partial P(W)$ are fixed under $P(f)$.
Case
{\rm (1)} corresponds infinitesimally to
{\rm (\ref{Pruefer_flow})}, while case {\rm (2)} induces a
flow $M(f)$ on the Moorisation $M(W)$ fixing the ``thorns''
point-wise.
\end{lemma}

\begin{proof}
We merely
define the Pr\"ufer flow $P(f)$. For a point in the interior
of $P(W)$, just let act the flow $f$. If instead the point is
a ``ray'', choose a tangent vector, represent it by a
path-germ, and let it evolve in time with the flow on $W$
(until the given time $t$ is elapsed), and take its tangent
vector to define the image ray.
\end{proof}

\smallskip
(18)
Since open metric implies a brush, any example in this class
must be compact.
For $g\ge 4$, $N_g$ (closed non-orientable surface of genus
$g$ with $\chi=2-g$) admits a {\it quasi-minimal
flow}\footnote{i.e., all orbits are dense, except for a finite
number of stationary points. (We follow the terminology of
Guti\'errez-Pires \cite{Gutierrez-Pires_2005} ({\it
supertransitive} or {\it highly transitive} are used in the
same or related contexts by other authors).} (see Guti\'errez
\cite[Prop.\,1]{Gutierrez_1978_JDE});
and so belongs to class (19). In contrast for $g=3$, it
is known that $N_3$ has no quasi-minimal flow
(cf. the discussion in Section~\ref{genus_3} below); so
belongs to class (18). In view of
the geographical location of all other closed surfaces (cf.
(19) for $\Sigma_g$, $g\ge 2$), it turns out that $N_3$ is the
unique representant in this class.

\smallskip
(19) Here we have the closed orientable surfaces $\Sigma_g$ of
genus $g\ge 2$ (with $\chi<0$). A quasi-minimal flow is
obtained, by expressing $\Sigma_g$ as a two-sheeted branched
covering of the 2-torus $\pi\colon \Sigma_g \to {\Bbb T}^2$,
and then lifting an irrational flow. This flow on $\Sigma_g$
has only dense orbits, except those corresponding to the
$2(g-1)$
ramification points of the map $\pi$ which are
saddle points.
As discussed in (18), the class (19) contains also the
surfaces $N_g$ with $g\ge 4$, yet not a single open surface.
Hence the class  (19) is also completely classified.

\smallskip
(20) As already discussed this class contains only the
projective plane ${\Bbb R}P^2$. 

\smallskip
(21) Here we meet the Klein bottle ${\Bbb K}$ (intransitive by
Lemma~\ref{Markley:prop}). In the compact case it is the only
example, for the only closed surfaces with a brush (hence
$\chi=0$) are ${\Bbb T}^2$ and ${\Bbb K}$, but the former is
transitive. What about non-compact examples? Using
Lemma~\ref{intransitivity} below, one can certainly take the
punctured Klein bottle.
Granting the {\it inflation principle},  any open set in
${\Bbb K}$ or ${\Bbb R}P^2$
is intransitive (otherwise inflate the set to be dense keeping
its homeomorphism type unchanged, and extend the transitive
flow to the ambient closed surface). Thus, all those open
sets
belong to this class provided not embeddable in ${\Bbb R}^2$
nor equal to ${\Bbb R}P^2$.

\smallskip
(22)
Such an example if it exists must be open. Further by
Beni\`ere's result  (Theorem~\ref{Beniere:thm} below) it must
be non-orientable. Since $N_3$ lacks a quasi-minimal flow
(Section~\ref{genus_3}), $N_3$ punctured once cannot be
quasi-minimal; so $N_3-\text{pt}$ belongs to (22). (Of course
the same applies to $N_3$ minus a finite set.)

\smallskip
(23) As before, an example if it exists must be open and
non-orientable. A candidate is $N_4$ punctured once.

\smallskip
(24) Here we have the ideal ``minimal'' dynamics. In the
compact case we have only ${\Bbb T^2}$ (for H.~Kneser 1924
\cite{Kneser24}  shows that any foliation on ${\Bbb K}$ has a
compact (hence circle) leaf).
Puncturing the torus one
finds many non-compact examples. The following result of
Beni\`ere~\cite{Beniere_1998} provides much more:

\begin{theorem}\label{Beniere:thm}  An open metric surface which is
orientable, yet not embeddable in ${\Bbb S}^2$ has a minimal
flow.
\end{theorem}

Lemma~\ref{intransitivity} below implies that the
orientability assumption cannot be relaxed in Beni\`ere's
result (consider a punctured
${\Bbb R}P^2$). (Caution: Beni\`ere's theorem is sometimes
quoted without the orientability proviso; compare
Nikolaev-Zhuzhoma \cite[p.\,xi,
p.\,252]{Nikolaev-Zhuzhoma_1999}.)
Nevertheless, a non-orientable surface
may well admit a minimal flow, as shown by Guti\'errez's
construction \cite{Gutierrez_1978_JDE} of a quasi-minimal flow
on $N_4$ (closed non-orientable surface of genus $4$) with two
hyperbolic saddles as unique singularities (other
non-stationary orbits are dense). Consequently $N_4-\{\text{2
pts}\}$ (two punctures)
has a minimal flow. (Thus, non-orientable manifolds may well
support minimal flows, answering partially a question of
Gottschalk \cite{Gottschalk_1958}. To get a compact example
one must in view of Kneser \cite{Kneser24} move to dimension
$3$, where one can suspend a minimal homeomorphism of the
Klein bottle constructed by Ellis.) A natural problem would be
a complete classification of surfaces with a minimal flow
(specialists are probably quite close to the goal?).

\smallskip
(25) This is the big remaining banlieue: we merely mention
$\Lambda_{g,n}$ for $g\ge 1, n\ge 1$.

\subsection{Transitivity transfers (up and down)}

This section exposes two basic results relevant to
our previous discussion (Section~\ref{Geography}). (We present
arguments in term of vector fields using the smoothing theory
of Guti\'errez, yet one could also work with $C^0$-flows using
Beck's technique, briefly discussed in
Section~\ref{Beck's_technique:section}.)

\begin{lemma}
{\bf (Transitivity preservation under finite puncturing).}
{\it Let $\Sigma$ be any closed transitive surface, then
$\Sigma$ punctured by a finite set $F$ is also transitive.}
\end{lemma}

\begin{proof} We may assume the flow smooth by
Guti\'errez 1986 \cite{Gutierrez_1986}. (On examples
smoothness is satisfied.) Consider $\xi$ the corresponding
velocity vector field.
Let $x\in \Sigma$ have a dense orbit.
We may perform the punctures outside this orbit. Take
$\varphi$ a non-negative smooth function on $\Sigma$ vanishing
exactly on $F$. Integrating the vector field $\varphi \xi$
yields a flow on $\Sigma-F$ whose orbit of $x$ remains dense,
indeed identic to the original trajectory.
\end{proof}

Here is a reverse engineering:

\begin{lemma}\label{intransitivity}
{\bf (Intransitivity preservation under closed
excision).} Assume the closed surface $\Sigma$ intransitive,
then $\Sigma-F$ ($F$ being an arbitrary closed (meagre) subset
of $\Sigma$)
is intransitive as well.
Granting the inflation conjecture {\rm
(\ref{Inflation:conj})}, any surface which embeds
in ${\Bbb S}^2$, ${\Bbb R}P^2$ or ${\Bbb K}$ is intransitive.
\end{lemma}

\begin{proof} Assume $\Sigma-F$ transitive under the
flow $f$ which we assume
smooth. (This involves the open case of Guti\'errez's
smoothing theory, alternatively use Beck's technique.)
Let
$x\in \Sigma-F$ be a point with dense orbit and let $\xi$ be
the
velocity vector field of the flow $f$. Choose $\varphi\ge 0$ a
non-negative $C^\infty$ function  on $\Sigma$ vanishing
exactly on $F$
(Whitney). Then the vector field $\varphi \xi$ on $\Sigma-F$
admits a smooth extension $\eta$ to $\Sigma$ vanishing on $F$.
Integrating this
field $\eta$ (over the compact manifold $\Sigma$) produces a
flow $f_{\eta}$ on $\Sigma$ such that the orbit of $x$ is
dense in $U=\Sigma-F$ (indeed identic to the original
trajectory of $x$), therefore dense in $\Sigma$.
\end{proof}

Thus, finitely punctured projective planes and Klein bottles
are still intransitive (hence belong to class (21) of the
previous section). (Sharper conclusions are
discussed in the next remark.) In particular ${\Bbb K}$ is not
the unique representant in its class. (Also this shows that
orientability is essential to Beni\`ere's result
(\ref{Beniere:thm}), e.g. the punctured projective plane (of
genus $1$) is intransitive.)

\begin{rem} {\rm We only
proved Lemma~\ref{intransitivity} under the assumption that
$F$ is meagre (empty interior) or what is the same if its
complement is dense. (For the application we made in (15)
above, this weak form was sufficient as the set ${\Bbb
R}V=f({\Bbb R}\times V)$
was dense.)
Yet, to sharpen the method, it
is desirable to
dispose of the {\it inflation principle}
(\ref{Inflation:conj}), to the effect that any non-void open
set of a (separable) manifold can be inflated to a dense
subset while keeping its homeomorphism type intact. Then any
open subset of these two surfaces ${\Bbb R}P^2$ or ${\Bbb K}$
is intransitive (hence belongs to class (21) provided it does
not embed into the plane and is not all ${\Bbb R}P^2$).}
\end{rem}

\subsection{Non-quasiminimality of
$N_3$ (Katok-Guti\'errez)}\label{genus_3}

This section---slightly outside of our main theme---can be
skipped without loosing continuity (its significance lies in
completing our understanding of the metric-side of our Venn
diagram in Section~\ref{Geography}).

Our interest  lies in the following proposition involving
primarily authors like Katok-Blohin, Guti\'errez and
Aranson-Zhuzhoma. The
proofs in the literature are oft
sketchy and in our opinion strangely cross-referenced. [For
instance the statement in Nikolaev-Zhuzhoma \cite[Lemma~7.4.1,
p.\,132]{Nikolaev-Zhuzhoma_1999}
may contain a minor bug\footnote{Since the surface $N_3$ is a
torus with one cross-cap, one can start with a Kronecker flow
on the torus, and deform it around the cross-cap, arranging
the speeds to vanish on the ``boundary'' of the cross-cap.
This gives a transitive flow on $N_3$ with a circle of fixed
points, which corrupts this Lemma~7.4.1. The latter seems
therefore implicitly formulated under the finiteness
assumption for the fixed-point set of the flow.}. This and
other sources (e.g. Aranson {\it et al.}
\cite{Aranson-Zhuzhoma-Telnykh_1998}) observe that the
assertion goes back to  Katok, as
reported in Blohin~\cite{Blohin_1972}, where unfortunately no
details are to be found.] Our argument
lacks in
rigor, yet we could not resist attempting a glimpse into the
boosted Poincar\'e-Bendixson theory of the aforementioned
authors. (Recall a flow is {\it quasi-minimal} if it has
finitely many stationary points and all non-stationary orbits
are dense.)

\begin{prop}\label{folklore_Russo-Peruvian} The closed non-orientable surface of genus $3$,
denoted by $N_3$, has no quasi-minimal flow.
\end{prop}

\begin{proof} By contradiction, assume $N_3$
equipped with a quasi-minimal flow.
The finitely many singular points all have a certain index.
Positive indices (in the form of {\it sources} or {\it sinks})
are forbidden as they both imply a small circular
cross-section enclosing the singular point, impeding
transitivity. The case of a {\it center} is likewise excluded.
Singularities of zero-indices (so-called {\it fake saddles})
are removable via a new flow {\it a fortiori} quasi-minimal.
Then all singular points have negative indices. The Poincar\'e
index formula
imposes, as $\chi(N_3)=-1$,
a unique singularity of index $-1$ (a {\it hyperbolic saddle}
with four separatrices).

A lemma of Peixoto-Guti\'errez \cite[Lemma~2,
p.\,312]{Gutierrez_1978_TAMS} gives a global cross-section $C$
to the flow.
This circle $C$ is {\it two-sided}, i.e. its tubular
neighbourhood (being oriented by the flow lines) is an annulus
(not a M\"obius band). Moreover $C$ is not dividing (a global
separation would impede transitivity). Cutting $N_3$ along the
curve $C$ yields a connected bordered surface $W$ with two
contours (boundary-components) with $\chi$ unchanged equal to
$-1$. Since the characteristic of a closed orientable surface
is even ($2-2g$ where $g$ is the genus), and since two
disc-excisions are required to create the two contours of $W$
it follows from the oddness of $\chi(W)$ that $W$ is
non-orientable. Hence $W$ is an annulus with one cross-cap.
The surface we started with, $N_3$, is recovered by gluing
back the two contours. Naively two sewing seem possible,
yet
indistinguishable as $W$ is
non-orientable. For psychological convenience, we fix
the
radial identification (between the two contours of the
annulus).

\begin{figure}[h]
\centering
    \epsfig{figure=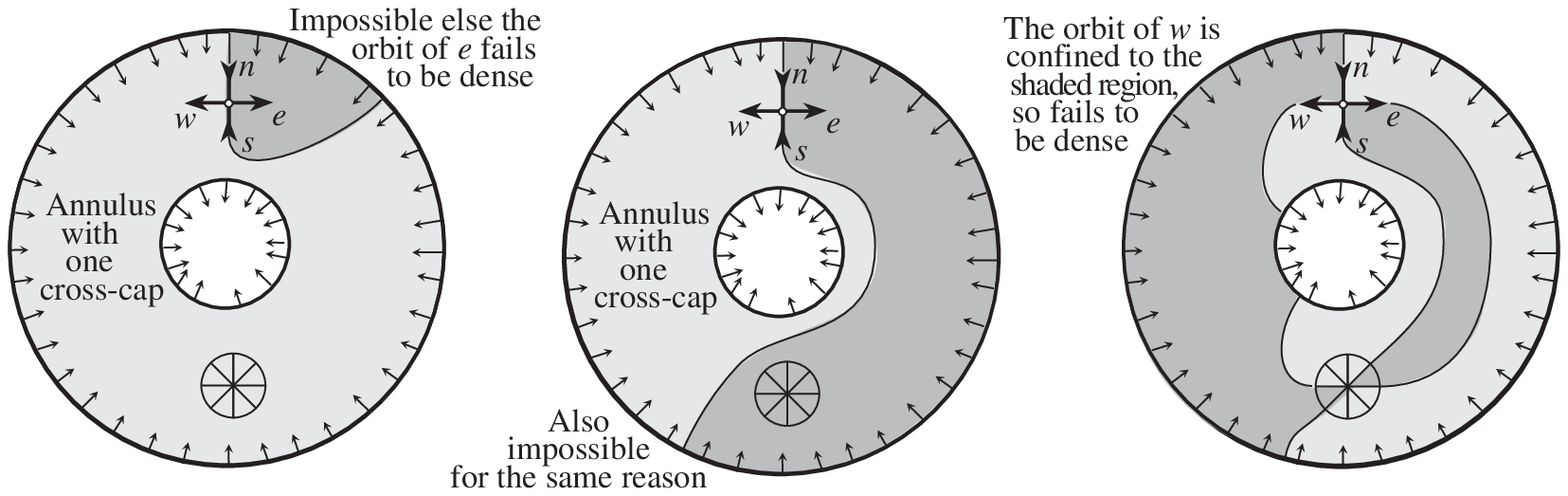,,width=122mm
    }
    \caption{\label{N_3} A heuristic Poincar\'e-Bendixson
    argument (\`a la Guti\'errez)}
\end{figure}

Figure~\ref{N_3} summarizes the situation: disjointly to the
cross-cap is drawn the unique hyperbolic saddle.
The flow is assumed  entrant on the
outer-boundary of $W$ and sortant on the inner-boundary. Let
the saddle be so oriented that its separatrices are directed
in the four cardinal directions North-West-South-East, say
with the North corresponding to an incoming (stable)
separatrix converging to the singularity. Link the north-point
$n$ to the outer-boundary (such a crossing must actually occur
since the orbit of $n$ is dense). The south-point $s$ cannot
move directly to the outer-contour (Figure~\ref{N_3},
left-side); this would impede either the east- or west-point
to fill a dense orbit. Similarly $s$ cannot reach the
outer-contour as on the center-part of Figure~\ref{N_3}, for
in this case $e$ is again trapped in the shaded sub-region. So
$s$ must travel through the cross-cap (Figure~\ref{N_3},
right-side). Draw the forward-orbit of $e$ until it reaches
the inner-contour (while traversing  the cross-cap), and
extend also the forward-orbit of $w$ until it intercepts the
inner-contour. Then the orbit of the west-point $w$ appears to
be
trapped in
the shaded sub-region delimited by
the
four
semi-orbits of the cardinal points (each
extended until its first-passage through the cross-section
$C$), and so fails to be dense. This contradiction provides a
vague completion of the proof. (A complete argument is
certainly implicit in
Guti\'errez~1978~\cite{Gutierrez_1978_TAMS}.)
\end{proof}

\subsection{Densification-Inflation conjecture}
\label{Inflation:sec}

The strong form of Lemma~\ref{intransitivity} (saying that
each open set of an intransitive closed surface is itself
intransitive) would be comforted
by the following:

\begin{conj}\label{Inflation:conj} (Inflation conjecture)  Assume that $M$ is a connected separable
manifold. Then for each non-empty open set $U$ in $M$, there
is an open set $V$ dense in $M$ and homeomorphic to $U$.
\end{conj}

The case where
$U$ is a chart holds true by virtue of the {\it phagocytosis
lemma} in \cite[Prop.\,1]{Gauld_2009}. Indeed, as $M$ is
separable, one may select a countable dense subset $D$ of $M$.
Like any countable subset of a connected (Hausdorff) manifold,
$D$ is contained in a chart $V$ (phagocytosis). This set $V$
fulfils the desiderata of
(\ref{Inflation:conj}).

As
expanded in the next section, an impetus for the {\it
Inflation Conjecture} arises in the construction of transitive
flows on manifolds. Yet, in view of the transitivity of the
number-spaces ${\Bbb R}^n$ $(n\ge 3)$, one can
often bypass the inflation principle to apply instead
phagocytosis.
%
%
In case there should be a failure of (\ref{Inflation:conj}),
then it may hold in special circumstances (like DIFF, metric,
compact, dimension two).
\subsection{Construction of transitive flows:
a
recipe}\label{const_trans_flows:gnal_recipe:sec}

A basic procedure to construct transitive flows on a
(separable) manifold $M$ is the following:

{\sc Step 1.} Find a {\it transistor} (or {\it transifold})
$T$, i.e. a manifold with a transitive flow (typically $T$
will be a torus undergoing some puncture or
the excision of some ``small'' closed set not
jeopardizing its transitivity). Examples are given below.

{\sc Step 2.} Embed (whenever it is possible) the transistor
$T$ in the given manifold $M$.

{\sc Step 3.} Using the
inflation principle (\ref{Inflation:conj})---alternatively
some {\it ad hoc}
construction---arrange the transistor $T$ to be densely
embedded in $M$.
(Sometimes phagocytosis acts as a substitute.)

{\sc Step 4.} ``Extend'' the flow to the full manifold $M$.
Then $M$ will be the desired transitive manifold.
The standard method uses vector fields (hence a smooth
structure), but there is also a $C^0$-version (Beck's
technique) working at least when the manifold is metric
(cf. Lemma~\ref{Beck's_technique:extension:lemma}). One hopes
however that a non-metric version holds, if not in full
generality at least
in special circumstances (maybe when the manifold has nice
functional properties).

In the surface-case (dimension $2$), a ``good'' transistor is
the punctured torus ${\Bbb T}^2_{\ast}={\Bbb T}^2-\{(0,0)\}$.
Subdivide the fundamental domain $[0,1]^2$ in
$3^2=9$ subsquares. Puncturing (the origin) amounts to
delete the $4$ peripheral subsquares, leaving
a ``Swiss cross'' with opposite edges identified, i.e. the
plumbing of two
(open) annuli.
This transistor embeds in
many surfaces (e.g., in all closed surfaces distinct from
${\Bbb S}^2$, ${\Bbb R}P^2$ and ${\Bbb K}$), and in fact in
all metric surfaces distinct from the latter plus their open
subsets. Besides, all strict open sets of ${\Bbb S}^2$ and
${\Bbb R}P^2$ embed in ${\Bbb K}$, for a punctured ${\Bbb
R}P^2$ is a M\"obius band which embeds in ${\Bbb K}$. Thus, we
recover the following classification of transitive metric
surfaces (compare
\cite{Jimenez_2004}):

\begin{prop} \label{trans-surfaces} A metric connected
surface is transitive if and only if it is not homeomorphic to
${\Bbb S}^2$, ${\Bbb R}P^2$ nor embeddable in the Klein bottle
${\Bbb K}$.
\end{prop}

\begin{proof} $\Rightarrow$
Otherwise,  using an inflation ({\ref{Inflation:conj}})
and an extension
violates the intransitivity of ${\Bbb K}$.

$\Leftarrow$ Such a surface contains a copy of the transistor
so apply the above ``4 steps'' recipe.
\end{proof}

In
dimension $3$, a {\it universal} transistor is the $3$-torus
${\Bbb T}^3={\Bbb R}^3/ {\Bbb Z}^3$ excised along the three
circles
axes $T={\Bbb T}^3-({\Bbb T}^1\!\! \times\!\! \{0\}\!\!
\times\!\! \{0\} \cup \{0\}\!\! \times\!\! {\Bbb T}^1\!\!
\times\!\! \{0\} \cup \{0\}\!\! \times\!\! \{0\}\!\!\times\!\!
{\Bbb T}^1)$. Universality of this transistor refers
to its embedability in Euclidean space ${\Bbb R}^3$ (hence in
all $3$-manifolds). [Indeed think of ${\Bbb T}^3$ as the cube
$[0,1]^3$ with opposite faces identified. Subdivide the
segment $[0,1]$ in $3$ subintervals, and accordingly the cube
in $3^3=27$ subcubes
({\it Rubik's cube}). Deleting the $3$ circles factor
amounts to
suppress all the ``peripheral'' subcubes of the Rubik's cube
(those with at least two visible faces)
leaving $1+6=7$ subcubes. Since opposite faces must be
identified, we get
(an open) cube with three handles, which embeds in ${\Bbb
R}^3$ without resistance.]
Thus, the above recipe
reproduces the following classic
result; compare Oxtoby-Ulam 1941 \cite{Oxtoby-Ulam_1941}
(compact polyhedrons of dimension $\ge 3$), Sidorov 1968
\cite{Sidorov_1968} (transitivity of ${\Bbb R}^n$, $n\ge 3$),
Anosov 1974 \cite{Anosov_1974} (ergodicity of smooth compact
manifolds $M^n$, $n\ge 3$):

\begin{prop}
Any metric connected $3$-manifold
is transitive.
\end{prop}


Eventually, the ultimate generalisation could be:

\begin{conj}\label{transitivity:conj} Any separable connected $3$-manifold (or of higher
dimensions) is transitive.
\end{conj}

This
sounds blatantly optimistic,
yet the only difficulties appear to be located in the
reparametrisation paradigm (i.e., Beck's technique briefly
discussed in the next section)---recall that in dimension $\ge
3$, the {\it inflation method} (\ref{Inflation:conj})
is superseded by {\it phagocytosis}
(since ${\Bbb R}^3$ is a transistor, e.g. by an {\it ad hoc}
inflation of the cube with $3$ handles, or via
Sidorov~\cite{Sidorov_1968}).

\subsection{Beck's technique (plasticity of flows)
}\label{Beck's_technique:section}

A
basic desideratum,
when dealing with flows, is
a two-fold
yoga
of ``restriction'' and ``extension'':

(1) {\it Given a flow on a space $X$ and an open subset
$U\subset X$,
find a flow on $U$ whose phase-portrait is the trace of the
original one}; and conversely:

(2)
{\it Given a flow on $U$,
find a flow on $X\supset U$ whose phase-portrait restricts to
the given one.}

\smallskip
Thus, one expects that any open set of a brushing is itself a
brushing, and that any separable super-space of a transitive
space is
likewise transitive, provided the sub-space is dense (or
becomes so, after a suitable inflation).

Problem (1) is
solved in Beck~\cite{Beck_1958}, when $X$ is metric.
(Example~\ref{Beck's_disruption} below indicates a non-metric
disruption.)
The same  technique of Beck (clever time-changes afforded by
suitable integrations),
solves Problem~(2) in the metric case (compare \cite[Lemma
2.3]{Jimenez_2004}):

\begin{lemma}\label{Beck's_technique:extension:lemma} Let $X$ be a locally compact metric
space and $U$ and open set of $X$. Given a flow $f$ on $U$,
there is a new flow $f^{\star}$ on $X$ whose orbits in $U$ are
identic to the one under $f$.
\end{lemma}

Tackling
Conjecture~\ref{transitivity:conj} seems to involve an
understanding of how much of the ``extended'' Beck technique
(2) holds non-metrically.
Even if
the full swing of
(2) should fail,
there is certainly much room for partial results, say for
separable $3$-manifolds with a civilised geometry in the large
(like 3D-avatars of the Pr\"ufer or Moore manifolds, or
perhaps those having sufficiently many functions, relating
perhaps the question to {\it perfectly normality}).

\begin{exam}\label{Beck's_disruption}
{\rm Let $X={\Bbb R}\times {\Bbb L}_+$ be equipped with the
natural (translation) flow $f$ along the first real factor,
and consider $U=X-F$ the open set residual to $F=\{0\}\times
\omega_1$. Then there is no flow $f_{\ast}$ on $U$ whose
orbit-structure is the restriction of the one of $f$ to $U$.
[Using the {\it ad~hoc} method of Section~\ref{ad_hoc:sec}, it
may be shown that the
surface $U=X-F$ is not a brushing.]}
\end{exam}

\begin{proof}
Assume by contradiction the existence of $f_{\ast}$.
Chronometer the time $\tau(x)$ required for a point $(-1,x)$
starting from the cross-section $\Sigma_{-1}:=\{-1\}\times
{\Bbb L}_+$ to reach $\Sigma_{1}:=\{1\}\times {\Bbb L}_+$
under $f_{\ast}$. This defines a function $\tau\colon {\Bbb
L}_+ \to {\Bbb R}_{\ge 0} \cup \{ \infty\}$ by letting
$\tau(x)=\infty$ if $x\in \omega_1$.
Post-composing
$\tau$ with an arc-tangent function (extended by mapping
$\infty$ to $\pi/2=:1$, for
simplicity!) gives a continuous function $\tau_{\ast}\colon
{\Bbb L}_+ \to [0,1]$. Thus, $\tau_{\ast}=1$ on $\omega_1$,
yet $<1$ outside, violating the fact that a real-valued
continuous function on the long ray
is eventually constant.
\end{proof}

\subsection{Morse-Thom
brushes for metric open $C^0$-manifolds, and Poincar\'e-Hopf}
\label{Morse-Thom:sec}

It is well-known, in the differentiable case at least, that an
open
metric manifold carries a critical-point free Morse function.
(This reminds us the name of Thom, but are unable to recover
our source!)
At any rate, a proof is provided in M.\,W.
Hirsch~\cite[p.\,571]{Hirsch_1961}, where rather the influence
of Henry Whitehead is
emphasized. In principle the result should extend to the
topological case, to give:

\begin{prop}\label{Morse-Thom} Any open
metric
$C^0$-manifold has a brush (where each orbit is a line).
\end{prop}

\begin{proof} By results of Morse, Kirby-Siebenmann
and Quinn---same references as in the proof of
\cite[Theorem~1.4]{BGG}---it is known
that there is a topological Morse function (the
$4$-dimensional case requires Quinn's smoothing of open
$4$-manifolds \cite{Quinn82}). One can
alter the Morse function to have no critical points (via the
usual trick of boring arcs, homeomorphic to $[0,\infty)$,
escaping to infinity\footnote{Laurent Siebenmann calls them
``ventilators''.}).
Like any submersion, this critical point free Morse function,
$f$, defines a codimension-one foliation, to which we may
apply Siebenmann's transversality \cite[Thm 6.26,
p.\,159]{Siebenmann72} to obtain a dimension-one foliation
transverse to the levels of $f$. The latter $1$-foliation is
clearly orientable (indeed oriented by increasing values of
$f$). By Whitney \cite{Whitney33} there is a compatible flow
for this foliation, which is the required brush. Conceivably,
a $C^0$-theory of gradient flows may bypass foliations---thus,
both Siebenmann and Whitney---yet going in the details will
probably involve a common soup of
technologies.
(Of course, such ``gradient'' flows are dynamically very
particular: each orbit of is a line (restrict
$f$ to the orbit), without ``recurrences'', and with plenty of
global cross-sections (any level-hypersurface of $f$.)
\end{proof}

Let us briefly discuss---without the pretention of proving
anything---a
possible relevance of (\ref{Morse-Thom}) to the  $C^0$-avatar
of {\it Heinz Hopf's brushes} (i.e., the hypothetical
existence of non-stationary flows on closed $C^0$-manifolds
with vanishing Euler character,
$\chi=0$). First, it is conceivable that to any
$C^0$-flow with isolated singularities on a closed
$C^0$-manifold  one may---despite the lack of vector field
interpretation---assign
{\it indices} (also via the Brouwer degree); compare the
procedure of \Kerekjarto \cite[p.\,109]{Kerekjarto_1925} in
the surface case, and also Dieudonn\'e
\cite[p.\,200]{Dieudonne}. Second, the {\it Poincar\'e-Hopf
index formula} is likely to hold, i.e. indices add up to the
characteristic of the manifold: like
by Italian geometers,
 why not just trying to
take advantage of the flow to push slightly the diagonal
$\Delta \subset M\times M$ into general position, to draw the
index formula from the two-fold evaluation (algebraic vs.
geometric) of the self-intersection
number $\Delta^2$. [If not really convinced---owing to a lack
of
foundations---translate
 the geometric
intuition into the cohomological language (of Moscow 1935:
Alexander-Kolmogoroff-Whitney-\v{C}ech).] This would
validate the index formula when $M$ is orientable, and the
general case follows by passing to the orientation covering.
Third, given any closed $C^0$-manifold, $M$, puncture it once
at a point, $p\in M$, to make it open and apply
Proposition~\ref{Morse-Thom} to get a nonsingular flow on
$M-\{p\}$. Using Beck's technique
(Lemma~\ref{Beck's_technique:extension:lemma}) the flow can be
extended to $M$ by fixing the point $p$. This would show that
{\it any closed $C^0$-manifold has a ``mono-singular'' flow},
i.e., with a unique rest-point (well-known in the
smooth case, \cite{Alexandroff-Hopf_1935}).
Finally, in case $\chi(M)=0$, one is tempted to claim that the
unique singular point, having zero index, is
removable. This would establish the desideratum.

{\small

}

{
\hspace{+5mm} 
{\footnotesize
\begin{minipage}[b]{0.6\linewidth} Alexandre
Gabard

Universit\'e de Gen\`eve

Section de Math\'ematiques

2-4 rue du Li\`evre, CP 64

CH-1211 Gen\`eve 4

Switzerland

alexandregabard@hotmail.com
\end{minipage}
\hspace{-25mm}
\begin{minipage}[b]{0.6\linewidth}
David Gauld

Department of Mathematics

The University of Auckland

Private Bag 92019

Auckland

New Zealand

d.gauld@auckland.ac.nz
\end{minipage}}

}


\begin{thebibliography}{30}

\bibitem{B} Mathieu Baillif, \textsl{The homotopy classes of continuous maps between some nonmetrizable manifolds}, Topology and its Applications, 148 (2005), 39--53.
\bibitem{BGG2} Mathieu Baillif, Alexandre Gabard and David Gauld,   \textsl{Dynamics of non-metric manifolds}, version of 11 November 2010.


\end{thebibliography}

\begin{thebibliography}{30}

\bibitem{Alexandroff-Hopf_1935}
P.~Alexandroff und H.~Hopf,  \textsl{Topologie},
Springer-Verlag, Berlin, 1935.

\bibitem{Anosov_1974}
D.\,V.~Anosov, \textsl{Existence of smooth ergodic flows on
smooth manifolds}, Math. USSR Izvestija 8 (1974), 525--552.


\bibitem{Aranson-Zhuzhoma-Telnykh_1998}
S.\,Kh.~Aranson, E.\,V. Zhuzhoma and I.\,A. Tel'nykh,
\textsl{Transitive and supertransitive flows on closed
nonorientable surfaces}, Math. Notes 63 (1998), 549--552.



\bibitem{BGG}
M.~Baillif, A. Gabard and D. Gauld, \textsl{Foliations on
non-metrisable manifolds: absorption by a Cantor black hole},
arXiv (2009).

\bibitem{BGG2}
M.~Baillif, A. Gabard and D. Gauld, \textsl{Foliations on
non-metrisable manifolds: Part II}, in preparation.

\bibitem{Beck_1958}
A.~Beck, \textsl{On invariant sets}, Ann. of Math. (2) 67
(1958), 99--103.


\bibitem{Bendixson1901}
I.~Bendixson, \textsl{Sur les courbes d\'efinies par des
\'equations diff\'erentielles}, Acta Math. 24 (1901), 1--88.

\bibitem{Beniere_1998}
J.\,C.~Beni\`ere, \textsl{Feuilletages minimaux sur les
surfaces non compactes}, Th\`ese d'\'etat, Universit\'e
Claude-Bernard, Lyon I, 1998.


\bibitem{Birkhoff_1936}
G.~Birkhoff, \textsl{A note on topological groups}, Compositio
Math. 3 (1936), 427--430.

\bibitem{Bing_1952} R.\,H. Bing, \textsl{A
homeomorphism between the $3$-sphere and the sum of two solid
horned spheres}, Ann. of Math. (2) 56 (1952), 354--362.

\bibitem{Blohin_1972} A.\,A. Blohin, \textsl{Smooth
ergodic flows on surfaces},
Trans. Moscow Math.
Soc. 27 (1972), 117--134.

\bibitem{Bohl_1904} P. Bohl, \textsl{\"Uber die Bewegung eines
mechanischen Systems in der N\"ahe einer Gleichgewichtslage},
J. reine angew. Math. 127 (1904), 179--276.

\bibitem{Brown_1962}
M.~Brown, \textsl{A mapping theorem for untriangulated
manifolds}, In: Topology of $3$-manifolds and related topics,
M.\,K. Fort, Jr., editor, Proc. of the Univ. of Georgia Inst.,
1961, 92--94. Prentice-Hall,  1962.

\bibitem{Brown_RF_1982}
R.\,F.~Brown, \textsl{The fixed point property and Cartesian
products}, Amer. Math. Monthly 89 (1982), no. 9, 654--678.

\bibitem{Calabi-Rosenlicht_1953} E. Calabi and M. Rosenlicht,
\textsl{Complex analytic manifolds without countable base},
Proc Amer. Math. Soc. 4 (1953), 335--340.



\bibitem{Cantor}
G.~Cantor, \textsl{Ueber unendliche, lineare
Punktmannichfaltigkeiten (5. Fortsetzung)}, Math. Ann.  {21}
(1883), 545--591.


\bibitem{Caratheodory_1932} C. Carath\'eodory,
\textsl{\"Uber die analytischen Abbildungen von
mehrdimensionalen R\"aumen}, in: Verhandlungen des
Internationalen Mathematiker-Kongresses Z\"urich 1932, I.
Band, Bericht und Allgemeine Vortr\"age, Orell F\"ussli
Verlag, Z\"urich und Leipzig, 1932, 93--101.

\bibitem{Caratheodory_1950} C. Carath\'eodory,
\textsl{Bemerkung \"uber die Definition der Riemannschen
Fl\"achen}, Math. Z. 52 (1950), 703--708.


\bibitem{Chewning_1974} W.\,C. Chewning,
\textsl{A dynamical system on $E^4$ neither isomorphic nor
equivalent to a differential system}, Bull. Amer. Math. Soc.
80 (1974), 150--153.

\bibitem{Dieudonne} J. Dieudonn\'e, \textsl{A History of Algebraic
and Differential Topology 1900--1960}, Birkh\"auser,
Boston-Basel, 1989.

\bibitem{Dyck_1888} W. Dyck,
\textsl{Beitr\"age zur Analysis situs. I Aufsatz. Ein- und
zweidimensionale Mannigfaltigkeiten}, Math. Ann. 32 (1888),
457--512.

\bibitem{Freedman82}
M.\,H.~Freedman, \textsl{The topology of four-dimensional
manifolds}, J. Differential Geom.  17 (1982), 357--453.

\bibitem{GaGa2010}
A. Gabard and D. Gauld, \textsl{Jordan and Schoenflies in
non-metrical analysis situs}, arXiv (2010).


\bibitem{Ganea_1954}
T. Ganea, \textsl{On the Pr\"ufer manifold and a problem of
Alexandroff and Hopf}, Acta Sci. Math. Szeged 15 (1954),
231--235.


\bibitem{Gauld_2009} D. Gauld, \textsl{Metrisability
of manifolds}, arXiv (2009).

\bibitem{Gauss_1839} C.\,F. Gauss,
\textsl{Allgemeine Theorie des Erdmagnetismus}, Resultate aus
den Beobachtungen des magnetischen Vereins im Jahre 1838,
Herausgb. v. Gauss u. Weber, Leipzig 1839. In: Werke, Bd. V,
pp.\, 119--193.



\bibitem{Gottschalk_1958} W.\,H. Gottschalk,
\textsl{Minimal sets: an introduction to topological
dynamics}, Bull. Amer. Math. Soc. 64 (1958), 336--351.

\bibitem{Gutierrez_1978_TAMS} C. Guti\'errez,
\textsl{Structural stability for flows on the torus with a
cross-cap}, Trans. Amer. Math. Soc. 241 (1978), 311--320.


\bibitem{Gutierrez_1978_JDE} C. Guti\'errez,
\textsl{Smooth nonorientable nontrivial recurrence on
two-manifolds}, J.
Diff. Equations 29 (1978), 388--395.


\bibitem{Gutierrez_1986} C. Guti\'errez,
\textsl{Smoothing continuous flows on two-manifolds and
recurrences}, Ergod. Th. Dynam. Sys. 6 (1986), 17--44.

\bibitem{Gutierrez-Pires_2005} C. Guti\'errez and B. Pires,
\textsl{On Peixoto's conjecture for flows on non-orientable
$2$-manifolds}, Proc. Amer. Math. Soc. 133 (2005), 1063--1074.

\bibitem{Hajek_1968}
O.~\Hajek\!\!, \textsl{Dynamical Systems in the Plane},
Academic Press, London and New York, 1968.

\bibitem{HectorHirschB}
G. Hector, U. Hirsch, {\em Introduction to the Geometry of
Foliations, Part B}, Aspects of Mathematics, Vieweg \& Sohn,
Braunschweig, 1983.


\bibitem{Hirsch_1961}
M.\,W.~Hirsch, \textsl{On imbedding differentiable manifolds
in Euclidean space}, Ann. of Math. (2)  73 (1961), 566--571.


\bibitem{Hopf_1927}
H.~Hopf, \textsl{Vektorfelder in $n$-dimensionalen
Mannigfaltigkeiten}, Math. Ann. 96 (1927), 225--250.

\bibitem{Jaworowski_1971}
J.\,W.~Jaworowski, \textsl{A fixed point theorem for
manifolds}, Proc. Amer. Math. Soc. 28 (1971), 275--278.


\bibitem{Jimenez_2004}
V.~Jim\'enez L\'opez and G. Soler L\'opez, \textsl{Transitive
flows on manifolds}, Rev. Mat. Iberoamericana 20 (2004),
107--130.



\bibitem{Jones_F._Burton_1966}
F.\,B.~Jones, \textsl{Metrization}, Amer. Math. Monthly  73
(1966), 571--576. (Especially p.\,573, where the Jordan curve
theorem is
formulated on the Moore surface.)


\bibitem{Jones_F._Burton_1997}
F.\,B.~Jones, \textsl{The beginning of topology in the United
States and the Moore School}, in: Handbook of the History of
General Topology, Vol.\,1, C.\,E. Aull and R. Lowen eds.,
97--103. Kluwer Academic Publishers, 1997.



\bibitem{Kerekjarto_1923}
B.~von \Kerekjarto\!\!, \textsl{Vorlesungen \"uber Topologie,
I, Fl\"achen Topologie}, Die Grundlehren der mathematischen
Wissenschaften in Einzeldarstellungen 96, Band VIII, Verlag
von Julius Springer, Berlin, 1923.

\bibitem{Kerekjarto_1925}
B. de \Kerekjarto\!\!, {\em On a geometric theory of
continuous groups. I. Families of path-curves of continuous
one-parameter groups of the plane}, Ann. of Math. (2) 27
(1925), 105--117.

\bibitem{Kneser24}
H.~Kneser, \textsl{Regul\"are Kurvenscharen auf den
Ringfl\"achen}, Math. Ann. 91 (1924), 135--154.

\bibitem{Lefschetz_1937}
S. Lefschetz, \textsl{On the fixed point formula}, Ann. of
Math. (2) 38 (1937), 819--822.

\bibitem{Lima_1964}
E.\,L. Lima, \textsl{Common singularities of commuting vector
fields on $2$-manifolds}, Comment. Math. Helv.  39 (1964),
97--110.


\bibitem{Markley_1969} N.\,G. Markley, \textsl{The
Poincar\'e-Bendixson theorem for the Klein bottle}, Trans.
Amer. Math. Soc. 135 (1969), 159--165.


\bibitem{Markoff_1931}
A. Markoff, \textsl{Sur une propri\'et\'e g\'en\'erale des
ensembles minimaux de M. Birkhoff}, C. R. Acad. Sci. Paris 193
(1931), 823--825.

\bibitem{Mather82}
J.\,N. Mather, \textsl{Foliations of surfaces: I, an ideal
boundary}, Ann. Inst. Fourier  32 (1982), 235--261.


\bibitem{Moebius_1863}
A.\,F. M\"obius, \textsl{Theorie der elementaren
Verwandschaft}, Ber. Verhandl. K\"onigl. S\"achs. Gesell. d.
Wiss., mat.-phys. Klasse 15  (1863), 18--57. (M\"obius Werke
II).

\bibitem{Moore_1942}
R.\,L. Moore, \textsl{Concerning separability}, Proc. Nat.
Acad. Sci. U.S.A. 28 (1942), 56--58.



\bibitem{Moore_1962}
R.\,L. Moore, \textsl{Foundations of Point Set Theory}, Amer.
Math. Soc. Colloquium Publications  {\it 13} (1962).
 Revised Edition of the 1932 original.
[See especially p.\,376--377. N.B. The
Moore surface
does not appear in the original 1932 edition,
though discovered in 1929 according to F. Burton Jones
\cite{Jones_F._Burton_1997}.]


\bibitem{Morrow_1969}
J.\,A. Morrow, \textsl{The tangent bundle of the long line},
Proc. Amer. Math. Soc. 23 (1969), 458--458.

\bibitem{Nemytskii-Stepanov} V.\,V. Nemytskii and
V.\,V. Stepanov, \textsl{The Qualitative Theory of
Differential Equations}, Dover 1989; republication of the
Princeton edition, 1960.

\bibitem{Nevanlinna_1953}
R. Nevanlinna, \textsl{Uniformisierung}, Springer, Berlin,
1953. [especially p.\,51--53]



\bibitem{Nikolaev-Zhuzhoma_1999}
I. Nikolaev, E. Zhuzhoma \textsl{Flows on $2$-dimensional
manifolds}, Lecture Notes in Math. 1705, Springer, Berlin,
1999.

\bibitem{Nyikos84}
P.\,J.~Nyikos, \textsl{The theory of nonmetrisable manifolds},
In: {\em Handbook of Set-Theoretic Topology}, ed.
 Kunen and Vaughan, 633--684, North-Holland,
Amsterdam, 1984.




\bibitem{Nyikos90}
P.\,J.~Nyikos, \textsl{On first countable, countably compact
spaces III: The problem of obtaining separable noncompact
examples}, In: {\em Open Problems in Topology}, ed.
 J. van Mill and G.\,M. Reed, 128--161, Elsevier
 Science Publishers B.V., North-Holland,
Amsterdam, 1990.



\bibitem{Oxtoby-Ulam_1941}
J.\,C.~Oxtoby and S.\,M. Ulam, \textsl{Measure preserving
homeomorphisms and metrical transitivity}, Ann. of Math. (2)
42 (1941), 874--920.



\bibitem{Peixoto_1962}
M.\,M.~Peixoto, \textsl{Structural stability on
two-dimensional manifolds}, Topology 1 (1962), 101--120.


\bibitem{Poincare_1885}
H.~Poincar\'e, \textsl{Sur les courbes d\'efinies par une
\'equation diff\'erentielle}, Journal de Math. (3) 7 (1881),
375--424; (3) 8 (1882), 251--296; (4) 1 (1885), 167--244; (4)
2 (1886), 151--217.


\bibitem{Poincare_1892}
H.~Poincar\'e, \textsl{M\'ethodes nouvelles de la m\'ecanique
c\'eleste}, Gauthier-Villars, 1892.




\bibitem{Quinn82}
F.~Quinn, \textsl{Ends of Maps. III: dimensions 4 and 5}, J.
Differential Geom. { 17} (1982), 503--521.

\bibitem{Rado_1923} T. Rad\'o,  \textsl{Bemerkung zur
Arbeit des Herrn Bieberbach: \"Uber die Einordnung des
Hauptsatzes der Uniformisierung in der Weierstra{\ss}sche
Funktionentheorie (Math. Annalen 78)}, Math. Ann. 90 (1923),
30--37.

\bibitem{Rado_1925}
T.~Rad\'o, \textsl{\"Uber den Begriff der Riemannschen
Fl\"ache}, Acta Szeged {2} (1925), 101--121.

\bibitem{Samelson_1965-homology} H. Samelson, \textsl{On Poincar\'e
duality}, J. Anal. Math.  14 (1965), 323--336.

\bibitem{Sidorov_1968}
E.\,A.~Sidorov, \textsl{Smooth topologically transitive
dynamical systems}, Mat. Zametki 4 (1968), 939--947.

\bibitem{Siebenmann72}
L.\,C.~Siebenmann, \textsl{Deformation of homeomorphisms on
stratified sets}, Comment. Math. Helv. 47 (1972), 123--163.

\bibitem{Spivak}
M. Spivak, \textsl{A comprehensive introduction to
differential geometry}, Vol. 1, Publish or Perish, New-York,
1970.

\bibitem{Veblen-Whitehead_1931}
O. Veblen, J.\,H.\,C. Whitehead \textsl{A set of axioms for
differential geometry}, Proc. Nat. Acad. Sci. U.S.A.  17
(1931), 551--560.



\bibitem{Vick} J.\,W. Vick, \textsl{Homology Theory: An
Introduction to Algebraic Topology}, Second Edition, Graduate
Texts in Mathematics { 145}, Springer-Verlag, 1994.

\bibitem{Weyl_1913} H. Weyl, {\it Die Idee der
Riemannschen Fl\"ache}, B.\,G. Teubner,
Leipzig, 1913.


\bibitem{Whitney33}
H. Whitney, \textsl{Regular families of curves}, Ann. of Math.
(2)  34 (1933), 244--270.

\bibitem{Whitney_1938}
H. Whitney, \textsl{Cross-sections of curves in $3$-spaces},
Duke Math. J. 4 (1938), 222--226. 

\end{thebibliography}
\end{document}